\documentclass[draftclsnofoot, onecolumn, journal]{IEEEtran}
\usepackage{amsmath}
\usepackage{amssymb}
\usepackage{amsthm}
\usepackage[shortlabels]{enumitem}
\usepackage{url}
\usepackage{color}
\usepackage{tikz}

\usetikzlibrary{shapes.misc,decorations.pathreplacing,decorations.markings,patterns,arrows.meta}

\newcommand{\NN}{{\mathbb N}}
\newcommand{\RR}{{\mathbb R}}

\newcommand{\cA}{{\mathcal{A}}}
\newcommand{\cB}{{\mathcal{B}}}

\newcommand{\cM}{{\mathcal{M}}}
\newcommand{\cF}{{\mathcal{F}}}
\newcommand{\cG}{{\mathcal{G}}}
\newcommand{\cP}{{\mathcal{P}}}
\newcommand{\cS}{{\mathcal{S}}}
\newcommand{\cU}{{\mathcal{U}}}
\newcommand{\cZ}{{\mathcal{Z}}}
\newcommand{\cX}{{\mathcal{X}}}
\newcommand{\cT}{{\mathcal{T}}}
\newcommand{\argmax}{\mathop{\mathrm{argmax}}}  
\newcommand{\sign}{\mathop{\mathrm{sign}}}

\makeatletter
\let\@fnsymbol\@arabic
\makeatother

\newtheorem{theorem}{Theorem}

\newtheorem{remark}[theorem]{Remark}
\newtheorem{corollary}[theorem]{Corollary}
\newtheorem{example}[theorem]{Example}
\newtheorem{problem}[theorem]{Problem}
\newtheorem{lemma}[theorem]{Lemma}
\newtheorem{definition}[theorem]{Definition}

\newtheorem{algorithm}[theorem]{Algorithm}

\title{Control-Oriented Learning on the Fly}
\date{}
\author{Melkior~Ornik, Arie~Israel, and~Ufuk~Topcu
\thanks{M.~Ornik is with the Institute for Computational Engineering and Sciences, University of Texas at Austin, Austin, TX, 78712 USA \quad e-mail: mornik@ices.utexas.edu.}%
\thanks{A.~Israel is with the Department of Mathematics, University of Texas at Austin, Austin, TX, 78712 USA \quad e-mail: arie@math.utexas.edu.}%
\thanks{U.~Topcu is with the Department of Aerospace Engineering and Engineering Mechanics, University of Texas at Austin, Austin, TX, 78712 USA \quad e-mail: utopcu@utexas.edu.}%
}

\begin{document}
\maketitle

\begin{abstract}
This paper focuses on developing a strategy for control of systems whose dynamics are almost entirely unknown. This situation arises naturally in a scenario where a system undergoes a critical failure. In that case, it is imperative to retain the ability to satisfy basic control objectives in order to avert an imminent catastrophe. A prime example of such an objective is the reach-avoid problem, where a system needs to move to a certain state in a constrained state space. To deal with limitations on our knowledge of system dynamics, we develop a theory of myopic control. The primary goal of myopic control is to, at any given time, optimize the current direction of the system trajectory, given solely the information obtained about the system until that time. We propose an algorithm that uses small perturbations in the control effort to learn local dynamics while simultaneously ensuring that the system moves in a direction that appears to be nearly optimal, and provide hard bounds for its suboptimality. We additionally verify the usefulness of the algorithm on a simulation of a damaged aircraft seeking to avoid a crash, as well as on an example of a Van der Pol oscillator.
\end{abstract}

\section*{Notice of Previous Publication}
This manuscript is a significantly extended version of M.~Ornik, A.~Israel, U.~Topcu, ``Myopic Control of Systems with Unknown Dynamics'' \cite{Ornetal17C}. In addition to providing the detailed proofs omitted from that paper for lack of space, this manuscript contains novel sections discussing the existence of a myopically-optimal control law, design of goodness functions, robustness to disturbances, promising directions of future work, and a simulation presenting a myopic approach to control of a Van der Pol oscillator. In particular, the content of Section \ref{measgoo}, Section \ref{nois}, Section \ref{damosc}, Section \ref{secfut}, Appendix \ref{exim}, and Appendix \ref{proofs} is entirely novel, and multiple other sections have been substantially expanded.

\section{Introduction}

\IEEEPARstart{I}n an event of a rapid and unexpected change in the dynamics of a system, the ability to retain certain degree of useful control over a system is crucial. A notable real-world example of such a catastrophic event is that of an aircraft losing its wing after collision with another airplane \cite{Alo06}. In the particular event, the pilot managed to retain enough control of the aircraft to be able to safely land at a nearby airbase. The pilot's strategy in this specific event depended on his intuition and prior experience at flying the aircraft. This paper seeks to develop a methodical approach to control of an unknown system in real time. The strategy we propose is based on an intuitive approach one might use when trying to drive an unknown vehicle: performing small ``wiggles'' in controls to see how the system behaves, before deciding on a longer-term control action.

In the described setting of control of an unknown system, the only data available to extract information on the system dynamics can be obtained during the system run. Data-driven learning of the dynamics of unknown systems has been a subject of significant recent research \cite{Bruetal16,Cubetal12,Hiletal15,SchLip09,SolOst08}. However, the developed methods are data-intensive or, alternatively, require some a priori knowledge of system dynamics; hence, they are not suitable for the above motivating example. A recent work \cite{Ahmetal17} based on a similar scenario uses differential inclusions to assess the safety of a trajectory governed by unknown dynamics. However, this work merely discusses uncontrolled dynamics, and does not prescribe a control law which would ensure safety of a trajectory. A number of works do deal with control design for unknown dynamics, with or without trying to explicitly learn a model \cite{CopRat94,Khaetal16,NaHer14,Yehetal90}. However, the results of those works again significantly differ from ours, as they generally do not include formal guarantees on short-time, non-asymptotic system performance. Additionally, the control objectives discussed in those papers are markedly different than the ones that naturally arise in the scenario we are exploring, as they are all based on reference tracking. In particular, \cite{CopRat94}, \cite{Khaetal16}, and \cite{Yehetal90} employ methods based on transfer functions to solve a tracking problem in an unknown system, while \cite{NaHer14} uses neural networks in optimal tracking control.

In contrast to the previously mentioned papers, we are primarily concerned with solving reach-avoid-type problems. This consideration is motivated by practical reasons: in the event of a system failure, satisfying advanced control specifications likely becomes impossible, and the system should concentrate on its survival. In physical systems, it is reasonable to pose the problem of system survival as a question of reaching a certain target set as soon as possible while avoiding obstacles and staying within a predefined safety envelope. In the motivating example of a damaged aircraft, which we will use as a running example throughout this paper, the target set is the airbase runway, and the aircraft should reach it as soon as possible without touching the ground beforehand.

A standard approach to reach-avoid problems in systems with known dynamics \cite{Zhoetal12} is based on constrained optimal control, where the constraints are given by the geometry of the safety envelope and obstacles, and the time to reach a target set is minimized. The solutions to this problem clearly depend on the full dynamics of the system. In our setting we have very little information on the system dynamics. In fact, we only have the information that can be obtained by observing a system trajectory since the beginning of the system run, i.e., since the abrupt change. Hence, it is unlikely that we can find the exact optimal control law to solve a reach-avoid problem.

We instead propose the notion of {\em myopic control}, and use it to suboptimally solve reach-avoid problems. In this notion, the system uses a control law that seems to work best {\em at the given time}, without any knowledge on whether that control law will lead to good results in the future. This framework is motivated by the inherent lack of firm knowledge of future system dynamics.

To counterbalance the lack of confidence about the future, a new myopically-optimal control law is calculated and applied every so often, after a short time interval. In order to determine the next myopically-optimal control law, the system continually modifies the previous law by adding a series of {\em wiggles}: small, short-time controls that do not significantly change the system trajectory, but act as a learning mechanism to determine the local control dynamics of the system. In order for this mechanism to function, we make a technical assumption that the underlying system is affine in control. We briefly discuss a potential modified strategy for the general case.

The proposed algorithm results in an approximate solution to the myopic optimal control problem, with a degree of suboptimality dependent on the length of the control law update interval and the size of learning controls, i.e., wiggles. Additionally, if there are known bounds on regularity of system dynamics, the parameters of the algorithm can be set to ensure any desired bound on the degree of suboptimality. Finally, since the algorithm does not use any a priori information on the system (except, perhaps, the aforementioned bounds on regularity), but learns its dynamics on the fly, it is entirely robust to disturbances, as it simply treats them as part of the original system dynamics.

The organization of this paper is as follows: Section \ref{probdef} formally introduces the class of problems we are interested in solving. Section \ref{mycont} then defines the myopic optimal control problem. This problem is founded on a notion of apparent {\em goodness} of a control law that attempts to satisfy a certain objective. Thus, Section \ref{measgoo} discusses design of an appropriate goodness function. Section \ref{algov} provides an overview of the proposed algorithm to solve the myopic optimal control problem. Section \ref{resb} then provides performance bounds for the proposed algorithm. In Section \ref{nois} we discuss the robustness of the algorithm in the presence of disturbances. Section \ref{simul} provides simulation results for the running example of a damaged aircraft, showing that the algorithm indeed performs as expected, as well as a highly nonlinear example of a Van der Pol oscillator. Substantial open questions identified in previous sections are discussed in Section \ref{secfut}, while Section \ref{conc} provides concluding remarks. The paper additionally contains two appendices. Appendix \ref{exim} contains a brief discussion of some sufficient conditions for the existence of a solution to the myopic optimal control problem. Finally, as the mathematical results used in Section \ref{resb} are rather technical, their full proofs are contained in Appendix \ref{proofs}.

\emph{Notation.} Set of all continuous functions from set $\cA\subseteq\RR^n$ to set $\cB\subseteq\RR^m$ is denoted by $C(\cA,\cB)$. For a set $\cA$, $2^\cA$ denotes the set of all subsets of $\cA$. The right one-sided derivative of function $f:I\to\RR^n$ at time $t$ is denoted by $df(t+)/dt$. For a vector $v\in\RR^n$, $\|v\|$ denotes its $2$-norm, and $\|v\|_\infty$ its max-norm. If $x,y\in\RR^n$, $x\cdot y$ denotes the standard dot product of $x$ and $y$. Unless stated otherwise, for $x\in\RR^n$, $x_i$, $i\in\{1,2,\ldots,n\}$, denotes the $i$-th component of $x$.

\section{Problem Definition}
\label{probdef}

We investigate the control system
\begin{equation}
\label{thesys}
\dot{x}=f(x)+\sum_{i=1}^m g_i(x)u_i
\end{equation}
which evolves on $\cX\subseteq\RR^n$, where $f,g_i\in C(\cX,\cX)$ are unknown functions, and $u=(u_1,\ldots,u_m)$ is the $m$-dimensional control input with values in the bounded set $\cU\subseteq\RR^m$. While it is trivial to modify our proposed algorithm for the case where $\cU$ is any manifold with corners in the sense of \cite{Joy12}, for the simplicity of our argument we take $\cU=\{u\in\RR^m~|~\|u\|_\infty\leq 1\}$. In the remainder of the text, a solution to \eqref{thesys} with control input $u$ and initial value $x_0$ is denoted by $\phi_u(\cdot,x_0)$.

We are assuming that, at any time instance, we are able to observe the full state $x$ and the entire control $u$. The case of partial observations is also of significant interest for physical systems. Even though it is not the primary focus of this paper, we briefly discuss it in Section \ref{secfut}. Additionally, in a practical sense of control of a real-world system, we are usually only interested in the behavior of a system on a compact state space $\cX$. It thus makes sense to assume that there exist $M_0\geq 0$ and $M_1\geq 0$ such that
\begin{itemize}
	\item $\|f(x)\|\leq M_0$, $\|g_i(x)\|\leq M_0$ for all $x\in\cX$, $i\in\{1,\ldots,m\}$,
	\item $\|f(x)-f(y)\|\leq M_1\|x-y\|$, $\|g_i(x)-g_i(y)\|\leq M_1\|x-y\|$ for all $x,y\in\cX$, $i\in\{1,\ldots,m\}$.
\end{itemize}

The class of control-affine systems covers a wide array of physical control systems, including standard linear aircraft models \cite{Bos92, Duketal88, Steetal16}, and a variety of nonlinear models of mechanical systems \cite{LaV06}, e.g., a unicycle with rider \cite{Navetal99}, quadrotor helicopter \cite{Moketal06}, and the F-8 Crusader aircraft \cite{GarJor77}. Control and optimal control of general control-affine systems have been substantially covered by previous literature (see, e.g., \cite{LaV06,BulLew04,Isi95,Son98}), including investigations of reach-avoid problems \cite{FadBro06}. However, motivated by the scenario of an unexpected failure in a physical system, our setting introduces two additional requirements on control design:

\begin{enumerate}[R1]
	\item{Functions $f,g_1,\ldots,g_m$ in \eqref{thesys} are unknown at the beginning of the system run. We are certainly allowed to use trajectory data during the system run to determine information on them, but when the run starts, we are aware merely that the system is of the form \eqref{thesys}, and potentially know a bound on the functions' norms $M_0$ and Lipschitz constants $M_1$.}
	\item{There is only one system run, starting from a single predetermined initial state $x_0$. All control design needs to be performed during that run, without any repetitions.}
\end{enumerate}

We note that the bulk of previous work on systems with disturbances or noise, as discussed in, e.g., \cite{DorBis11, KwaSiv72}, does not help to deal with requirement R1. The standard setup in the theory of systems with disturbances requires the controller to know the ``general'' system dynamics, which are then modified by some unknown, but bounded, disturbance or noise. Requirement R1 is significantly more limiting for control design. It stipulates that the system dynamics are entirely unknown at the beginning of the system run. Potential additional disturbances to system dynamics in our setup are discussed in Section \ref{nois}.

As previously mentioned, the goal of this paper is to develop a method for solving reach-avoid-type control problems \cite{Esfetal16,Lev11} in the control-affine setting \eqref{thesys} under the above requirements R1 and R2. We use the following formulation for the reach-avoid problem:

\begin{problem}[Reach-avoid Problem]
	\label{rap}
	Let $x_0\in\cX$, $T\geq 0$, $\cT\subseteq\cX$, and $\cB\subseteq\cX$. Let $\phi_u(\cdot,x_0):[0,T]\to\cX$ be the trajectory generated by system \eqref{thesys} under control $u$ and with $\phi_u(0,x_0)=x_0$. Find, if it exists, the control law $u^*:[0,T]\to\cU$ such that the following holds:
	\begin{enumerate}[(i)]
		\item $\phi_{u^*}(t,x_0)\notin\cB$ for all $t\in[0,T]$,
		\item there exists $0\leq T'_{u^*}\leq T$ such that $\phi_{u^*}(t,x_0)\in\cT$ for all $t\in[T'_{u^*},T]$, and
		\item $T'_u$ from (ii) is the minimal such value for all the control laws $u:[0,T]\to\cU$ which satisfy (i) and (ii). 
	\end{enumerate}
\end{problem}

Let us build intuition through a simplification of Problem \ref{rap}. We note that if there is no particular target set $\cT$, the reach-avoid problem naturally transforms into the problem of remaining outside the undesirable set for as long as possible. This problem is again one natural abstraction of our running example: instead of having to deal with the technicalities of aircraft landing, we may simply want to pose the question of requiring the damaged aircraft to stay above ground for as long as possible. Let us provide a formal description of this version of the problem. 

\begin{definition}
	Let $x_0\in\cX$, and let $\cB\subseteq\cX$ be closed. Using the notation of Problem \ref{rap}, the {\em first bad time} of the trajectory $\phi_u$ is given by $$b(u)=\min_{t\geq 0}\{\phi_u(t,x_0)\in\cB\}\textrm{.}$$ If $\phi_u(t,x_0)\notin\cB$ for all $t\geq 0$, we take $b(u)=\infty$.
\end{definition}

\begin{problem}[Avoid Problem]
	\label{pro1}
	Let $x_0\in\cX$, and let $\cB\subseteq\cX$ be closed. Find a control law $u^*:[0,+\infty)\to\cU$ such that $$b(u^*)=\max_{u}b(u)\textrm{.}$$
\end{problem}

This paper would ideally end by providing a control design to solve Problem \ref{rap} and Problem \ref{pro1} which respects requirements R1 (unknown system) and R2 (no repeated runs). However, this outcome is likely impossible. The requirements we have imposed imply that, at the beginning of the system run, we have no particular knowledge of the values of functions $f(x)$ and $g_i(x)$ at any point $x$, and no way of acquiring significant knowledge about the dynamics in a certain area of $\cX$ until the system trajectory reaches that area. In fact, even when the system trajectory reaches a point $x$, we are still unable to exactly determine the values of $f(x)$ and $g_i(x)$, because a single controlled trajectory provides only limited information on full control dynamics around a point. Thus, it is not reasonable to expect that Problem \ref{rap} and Problem \ref{pro1} can be exactly solved under the requirements of our setting. We now present an intuitive variation of these problems that can be approached using strategies which respect requirements R1 and R2.

\section{Myopic Control}
\label{mycont}

Suppose that we want to solve the reach-avoid or avoid problem under requirements R1 and R2. Requirement R2 implies that, in order to determine the optimal control law $u^*$ to be used starting at some time $t$, one may only use the information obtained from the system trajectory until time $t$. Thus, based on our knowledge of the system dynamics in the interval $[0,t]$, one would have to assess the ``quality'' of a candidate future control law on the entire interval $(t,+\infty)$. However, we can only do so with very limited confidence, since the dynamics far away from points already visited at some time $0\leq t'\leq t$ are entirely unknown. Thus, at time $t$ we have no knowledge with certainty on whether a particular control will result in good behavior at time $t+\alpha$, for some large $\alpha$. 

The above discussion behooves us to replace Problem \ref{rap} and Problem \ref{pro1} by a {\em myopic optimal control problem}, where we want to design a control law such that, at every time instance, the trajectory {\em appears} to behave as well as possible. For instance, if our ultimate desire is to solve Problem \ref{pro1}, it makes sense to require that the trajectory at any given instance of time is moving away from the undesirable set $\cB$ as fast as possible. Clearly, this reasoning may not be ultimately optimal, as moving away from $\cB$ as fast as possible at some point could bring the trajectory into a position where the dynamics are such that it is simply forced to enter $\cB$ (see Figure \ref{figexp} for an example). Thus, this notion of ``appearing to behave as well as possible'' is simply our best guess. Since we do not have (almost) any knowledge of the dynamics far away from our current position, the best we can do is try to find a seemingly optimal direction of movement at any given time, given the totality of prior information obtained during the system run until that point.

\begin{figure}[ht]
	\centering
	\begin{tikzpicture}[yscale=0.5]
	\tikzstyle{point}=[thick,fill=red, draw=red,circle,inner sep=0pt,minimum width=4pt,minimum height=4pt]
    
    \fill[lightgray] (-4.5,0.5) -- (3,0.5) -- (3,-0.5) -- (-4.5,-0.5);
	\draw[thick, blue] (2,0) arc(0:120:2);
    \draw[thick, blue] (-4,3.464) arc(180:300:2);
	\node[label={[label distance=0mm]270:$\cB$}] at (0,0.7) {};
 	\node[label={[label distance=-1mm]305:$x$}] at (-1,1.732) {};
	\draw[red,-{Latex[length=2mm]}] (-1,1.732) -- (0,2.232);
	\draw[green,-{Latex[length=2mm]}] (-1,1.732) -- (-2,1.232);
	\end{tikzpicture}
	\caption{An example illustrating the built-in imperfection of myopic control. Assume that the system dynamics are such that, unbeknownst to the control designer, the system is restricted to moving along the blue curve. The control objective is to avoid the undesirable set $\cB$, denoted in gray. When the system is at point $x$, the available direction vectors are denoted by the green and red arrows. From the designer's myopic perspective, it may appear better to control the system in the red direction, because this direction forces the system to move further away from $\cB$. That action will ultimately result in the system moving into the undesirable set. Given the designer's lack of knowledge about the system, such a consequence is unforseeable.}
	\label{figexp}
\end{figure}
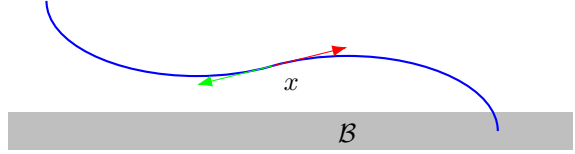

We propose to formalize the above notion of ``appearing to behave as well as possible'' using a {\em goodness function}
\begin{equation}
\label{defg}
(\phi,v)\mapsto G(\phi,v)\textrm{,}\,\quad \phi\in\cF,v\in\RR^n\textrm{,}
\end{equation} 
where 
\begin{equation*}
\cF=\bigcup_{T\geq 0}C([0,T],\cX)\textrm{.}
\end{equation*}
Function $G$ is intended to quantify how well the trajectory appears to be doing at satisfying its control specifications at a time when its trajectory until time $T$ and velocity at time $T$ are given by $\phi$ and $v$, respectively. 

Myopic policies and goodness functions have long been used in decision theory, in a wide variety of settings, including resource allocation \cite{Pow11}, computer vision \cite{RimBro94}, and Bayesian learning \cite{ZhaYu13}. However, all these settings differ significantly from the setting of this paper; to the authors' knowledge, myopic decision-making has never been utilized in dealing with control of unknown systems.

\begin{example}
\label{ex}
Going back to our running example of a damaged aircraft, assume that we want to solve Problem \ref{pro1}, i.e., have the aircraft stay away from the ground for as long as possible. In that case, if $x_1$ and $x_2$ denote the aircraft's horizontal and vertical position, respectively, $\cB\subseteq\RR^2$ is given by $\cB=\{(x_1,x_2)~|~x_2\leq 0\}$. Then, one natural option is for $G$ to measure the slope on which the trajectory is moving towards the boundary of $\cB$, i.e., $x_2=0$. Without any additional information on the system, the more negative this slope is, the worse the aircraft naturally appears to be doing. Hence, $$G(\phi,v)=\frac{v_2}{v_1}\textrm{,}$$ where we disregard the case of $v_1\leq 0$, is a reasonable goodness function. Figure \ref{fig1} provides an illustration of this goodness function.

\begin{figure}[ht]
\centering
\begin{tikzpicture}[yscale=0.5]
    \tikzstyle{point}=[thick,fill=black, draw=black,circle,inner sep=0pt,minimum width=2pt,minimum height=2pt]
		\node[point, label=left:{$\phi(T)$}] at (0,2) {};
		\node[label=right:{$G(\phi,v)=0$}] at (2,2) {};
		\node[label=right:{$G(\phi,v)=0.5$}] at (2,4) {};
		\node[point, label=below:{$G(\phi,v)=-2$}] at (0.5,0) {};
		\node[point, label=below:{$G(\phi,v)=-0.25$}] at (4,0) {};
		\draw[dashed] (0,2) -- (2,2);
		\draw[dashed] (0,2) -- (2,4);
		\draw[dashed] (0,2) -- (4,0);
		\draw[dashed] (0,2) -- (0.5,0);
		\draw (-2,0)--(5,0);
\end{tikzpicture}
\caption{An illustration of the goodness function $G$ from Example \ref{ex}. Dashed lines represent the possible tangents to the trajectory of system \eqref{thesys} at $\phi(T)$, and thus, our local predictions of future movement of the system state. These tangents are evaluated by $G$ depending on their slope towards the undesirable set $\cB$, whose edge is represented by a solid line.}
\label{fig1}
\end{figure}
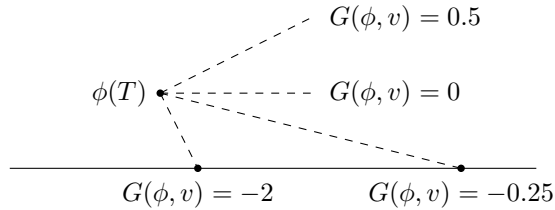
\end{example}

In general, determining an appropriate goodness function depends on our information on the system and the control specifications. We briefly discuss this problem, and provide further examples, in Section \ref{measgoo}.

Having discussed how a goodness function provides a measure of apparent quality of a trajectory at a given point (with respect to, for instance, specifications of Problem \ref{rap} and Problem \ref{pro1}), we now introduce the myopic optimal control problem. As a slight deviation from standard notation, in order to emphasize that $\phi\in\cF$ has $[0,T]$ as its domain, we will occasionally write such functions as $\phi|_{[0,T]}$.

\begin{problem}[Myopic Optimal Control Problem]
\label{pro2}
Let $x_0\in\cX$. Find a control law $u^*:[0,+\infty)\to\cU$ such that, for all $t\geq 0$, if $x=\phi_{u^*}(t,x_0)$, then 
\begin{equation}
\label{maxg}
G \left(\phi_{u^*}(\cdot,x_0)|_{[0,t]},\frac{d\phi_{u^*}(0+,x)}{dt}\right) = \max_{u\in\cU} G\left(\phi_{u^*}(\cdot,x_0)|_{[0,t]},\frac{d\phi_u(0+,x)}{dt}\right)\textrm{.}
\end{equation}
\end{problem}
We note that in \eqref{maxg} we are slightly abusing notation: if $u\in \cU$, then $u$ is not formally a control input as a function of time. Thus, by $d\phi_u(0+,x)/dt$ we denote the value of the right-hand side of \eqref{thesys} for the particular $x\in X$ and $u\in \cU$. For the sake of narrative, we will usually omit a formal distinction between a control law $u$ and a control input $u\in \cU$ when such a distinction is obvious from the context. Additionally, while the notion of a solution to \eqref{thesys}, and hence to \eqref{maxg} as well, is questionable in the case of general measurable functions $u^*$, in the remainder of the paper we will be working with piecewise-continuous control laws, and in that sense we will not have any issues.

In subsequent sections of this paper, we will propose an approximate piecewise-constant solution to Problem \ref{pro2}. This solution can be made arbitrarily close to optimal, in the sense that that for any $T>0$ and $\mu>0$, Algorithm \ref{alg} can generate a piecewise-constant control law that satisfies 
\begin{equation}
\label{magsu}
\bigg| G\left(\phi_{u^*}(\cdot,x_0)|_{[0,t]},\frac{d\phi_{u^*}(0+,x)}{dt}\right) - \max_{u\in\cU} G\left(\phi_{u^*}(\cdot,x_0)|_{[0,t]},\frac{d\phi_u(0+,x)}{dt}\right)\bigg|<\mu
\end{equation} for all $t\geq T$. We note that this result does not directly guarantee existence of an {\em exact} solution to Problem \ref{pro2}. Some sufficient conditions for the existence of such a solution can be obtained. Nonetheless, because of our inability to exactly learn the system dynamics at even a single point from a single trajectory, it is not realistic to expect that any method of finding an exact solution to Problem \ref{pro2} would respect requirements R1 and R2. Discussing such solutions is thus not in the focus of this paper. However, we refer the reader to a brief discussion of sufficient conditions for the existence of a solution to Problem \ref{pro2} contained in Appendix \ref{exim}.

Before proposing a strategy to find an approximate piecewise-constant solution to Problem \ref{pro2} (while respecting requirements R1 and R2), let us briefly return to the issue of determining an appropriate goodness function, given the control objective and system information.

\section{Measuring Goodness}
\label{measgoo}

Determining an appropriate goodness function $G$ is generally not easy. We provided a very simple possible goodness function for one scenario in Example \ref{ex}. Let us provide three more examples of natural goodness functions, two for Problem \ref{rap} and one for Problem \ref{pro1}.

\begin{example}
	\label{expro1}
As the sole objective of Problem \ref{pro1} is not to enter the bad set $\cB$, it makes sense to desire the trajectory to be moving away from $\cB$ as fast as possible. Without any additional knowledge, out best approximation of the system state at time $\Delta t$ after it is at point $x$ is $x+v\Delta t$, where $v$ is the velocity vector at the given time. Hence, if $d_\cB:\cX\to[0,+\infty)$ denotes the usual distance from a point to set $\cB$, one reasonable option for $G$ is to take $$G\left(\phi|_{[0,t]},v\right)=\liminf_{\Delta t\to 0+}\frac{d_\cB\left(\phi(t)+v\Delta t\right)-d_\cB\left(\phi(t)\right)}{\Delta t}\textrm{.}$$
This approach to a goodness function is similar to the one developed in \cite{AlvMar11} in the context of viability theory, where the distance from the undesirable set is considered as an indicator of a trajectory robustness. Since in our case we are interested in measuring the ``robustness'' (i.e., goodness) of our prediction of the future trajectory, we replace the distance from the undesirable set by the rate of change of the distance. 
\end{example}

\begin{remark}
	\label{exrem}
	We note that function $G$ in Example \ref{expro1} is not a generalization of the goodness function from Example \ref{ex}: adapted to set $\cB$ from Example \ref{ex}, the goodness function from Example \ref{expro1} would be given by $v_2$, instead of $v_2/v_1$ as in Example \ref{expro1}. Both of these goodness functions certainly possess a reasonable rationale: while the function from Example \ref{ex} measures how far we expect the trajectory to get before it hits $\cB$, function from Example \ref{expro1} measures how fast we expect the trajectory to hit $\cB$. This difference between two different intuitively reasonable goodness functions for the same problem serves to confirm our previous point that it is difficult to develop the notion of an objectively correct goodness function. As mentioned, while we propose several goodness functions for the reach-avoid and avoid problems in this section, a formal investigation of a notion of a correct goodness function is not a focus of this paper.
\end{remark}

The goodness function proposed in Example \ref{expro1} will form the basis of the goodness function for the example of a damaged aircraft in Section \ref{damair}. However, because of the particularities of aircraft design, the control specifications in that example are more complicated than just stating that the aircraft must stay at a positive altitude. Additionally, due to the limitations of aircraft design, certain states cannot be directly influenced by control actions. Hence, the goodness function in Section \ref{damair} is more difficult than the one presented in Example \ref{expro1}. While we refer the reader to Section \ref{damair} for details, we emphasize that the goodness function in Example \ref{expro1} provides a basis for the developments of that section. We now proceed to propose a goodness function for the full reach-avoid problem.

\begin{example}
	\label{exrap}
	Framework of Problem \ref{rap} is significantly more complicated than Problem \ref{pro1}, as there are two objectives that need to be satisfied at the same time. One naive reasonable goodness function uses the following motivation: let us partition the safe set $\cX\backslash\cB$ into a ``boundary zone'' $\cZ^B$ and a ``interior zone'' $\cZ^I$. The motivation for naming of the two sets is that, if a system state is inside $\cZ^B$, it is considered at risk of leaving the safe set, and the primary objective becomes to enter the safer interior zone as soon as possible. If a system state is inside the interior zone, it is not considered at risk, and the primary objective is to reach the target set $\cT$. Then, a possible goodness function is given by $$G\left(\phi|_{[0,t]},v\right)=
	-\liminf\limits_{\Delta t\to 0+}\frac{d_{\cZ^I}\left(\phi(t)+v\Delta t\right)-d_{\cZ^I}\left(\phi(t)\right)}{\Delta t}$$ if $\phi(t)\in\cZ^B$, and $$G\left(\phi|_{[0,t]},v\right)=-\liminf\limits_{\Delta t\to 0+}\frac{d_\cT\left(\phi(t)+v\Delta t\right)-d_\cT\left(\phi(t)\right)}{\Delta t}$$ if $\phi(t)\in\cZ^I$.
Again, the above goodness function is not the only reasonable one, and has its drawbacks. For instance, it is not continuous in $\phi(t)$, which will become important when determining bounds for suboptimality of a solution to Problem \ref{pro2} in Section \ref{resb}. A more subtle goodness function may use a measure of how much a system state is at risk (for instance, distance from $\cB$), and then use a mixed approach between trying to enter set $\cT$ and not leave $\cX\backslash\cB$:
\begin{equation*}
G\left(\phi|_{[0,t]},v\right)=  (1-\lambda)\liminf_{\Delta t\to 0+}\frac{d_{\cB}\left(\phi(t)+v\Delta t\right)-d_{\cB}\left(\phi(t)\right)}{\Delta t} - \lambda\liminf_{t\to 0+}\frac{d_\cT\left(\phi(t)+v\Delta t\right)-d_\cT\left(\phi(t)\right)}{\Delta t}\textrm{,}
\end{equation*}
	where $$\lambda=\frac{d_\cB\left(\phi(t)\right)}{\max_y d_\cB(y)}\textrm{.}$$	
\end{example}

As the above examples show, there are multiple reasonable goodness functions for a particular control objective, and the choice of which function to use rests with the designer. Let us now briefly mention a particularly attractive facet of this freedom to choose a goodness function. The choice of goodness functions is a natural vessel for the incorporation of any kind of side information on the system. For instance, if we were to know that the control dynamics evolve on a particular subset of $\cX$ instead of entire $\cX$, it might make sense to choose the goodness function in a way that is conscious of such information. For example, Figure \ref{unnatman} illustrates a situation where the system is known to evolve on a particular subset of $\RR^2$, and while the naive goodness function would choose a direction that points towards the desired target set, but ultimately leads nowhere, an informed goodness function would choose a direction which seems counterintuitive, but ultimately leads towards the target set.

\begin{figure}[ht]
	\centering
	\begin{tikzpicture}[yscale=0.5]
	\tikzstyle{point}=[thick,fill=red, draw=red,circle,inner sep=0pt,minimum width=4pt,minimum height=4pt]
	\draw[thick, blue] (0,0) circle(2);
	\draw[thick, blue] (2,0) arc(-90:90:2);
	\node[point, label=left:{$\cT$}] at (2,4) {};
	\node[label=right:{$\cM$}] at (1.414,-1.414) {};
	\node[label={[label distance=-1mm]180:$x$}] at (2,0) {};
	\draw[red,-{Latex[length=2mm]}] (2,0) -- (2,1);
	\draw[green,-{Latex[length=2mm]}] (2,0) -- (3,0);
	\end{tikzpicture}
	\caption{An example of the role of side information when designing a goodness function. The control objective is to reach the set $\cT$, and system trajectories are, unlike in Figure \ref{figexp}, a priori known to evolve on the set $\cM\subseteq\RR^2$ (drawn in blue). Assume that the system trajectory is currently (at time $t$) at position $x$, and available velocity vectors at that point are given by the red and green arrows. While an uninformed goodness function might prefer the red arrow, which points directly towards $\cT$, the available side information urges the designer to apply such a goodness function that the trajectory continue in the direction indicated by the green arrow.}
	\label{unnatman}
\end{figure}
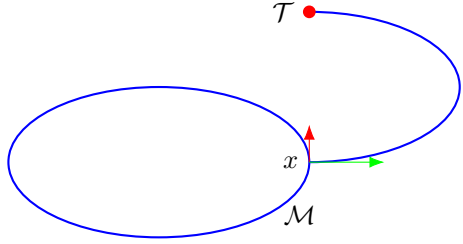

\begin{remark}
In all of the above examples, the proposed goodness functions are {\em memoryless}: $G(\phi_{[0,T]},v)$ is actually merely a function of $(\phi(T),v)$. However, the proposed theory allows for more complicated measures of goodness, incorporating knowledge obtained earlier in the system run. For instance, the fact that a goodness function is permitted to be memory-conscious allows to directly deal with temporal logic specifications. We consider linear temporal logic (LTL) specifications given over a collection $\cP=\{\cA_k~|~k\in K\}$ of subsets of $\cX$. An example of such a specification is ``visit $\cA_3$ after $\cA_2$, and do not enter $\cA_4$ before visiting $\cA_1$''. We omit a formal description of LTL, and direct the reader to \cite{ManPnu92}. Informally, an LTL formula determines all {\em allowed discrete behaviors} of a trajectory, where the discrete behavior of a trajectory describes the order in which a system trajectory visits the subsets in $\cP$. For any desired discrete behavior, the problem of finding a control law such that the system trajectory has that behavior can be posed as a series of reach-avoid problems \cite{Papetal16}. In other words, the system is first required to reach some target set $\cT_1$ while avoiding an undesirable set $\cB_1$, then required to reach set $\cT_2$ while avoiding $\cB_2$, etc. Guided by our previous examples, a reasonable goodness function to solve such a problem is given by $G(\phi|_{[0,t]},v)=G_i(\phi|_{[0,t]},v)$, where $i\in\NN$ is the smallest index such that $\{\phi|_{[0,t]}(\tau)~|~\tau\in[0,t]\}\cap\cT_i=\emptyset$, and $G_i$ is a goodness function for Problem \ref{rap} with target set $\cT_i$ and undesirable set $\cB_i$ (for instance, $G_i$ can be set to one of the two goodness functions proposed in Example \ref{exrap}). 
\end{remark}

The development of general goodness functions is not in the focus of this paper. However, the examples above show that goodness functions can be successfully designed for a variety of reach-avoid-type problems. We now turn towards our proposed algorithm for the approximate solution of Problem \ref{pro2}.

\section{Learn-Control Algorithm}
\label{algov}

We intend to approximately solve Problem \ref{pro2} by the use of short-time piecewise-constant controls. The proposed strategy relies on repeatedly learning the local dynamics at points on the observed system trajectory, and using those dynamics to simultaneously apply a myopically-optimal control law.

In order to learn the local dynamics around a single point on the trajectory, the controller can apply any $m+1$ affinely independent constant controls for a short period of time. Since system \eqref{thesys} is control-affine, we can use the observed changes in states to approximate the function $u\mapsto v_x(u)$ defined by $$v_x(u)=f(x)+\sum_{i=1}^m g_i(x)u_i\textrm{,}$$ where $x$ is a point on the trajectory at the end of this learning process.

After obtaining an approximation for $v_x$, if $\phi$ denotes the trajectory until the end of the learning process, by determining $\argmax_u G(\phi,v_x(u))\textrm{,}$ a constant control $u^*$ that appears to best satisfy the control specifications at point $x$ can be found. This control is then used during the next iteration of the learning process. In other words, the next learning process uses an affinely independent set of controls $u^*+\Delta u^0,\ldots,u^*+\Delta u^m$, where $\Delta u^i$ are small enough to ensure that their application still results in near-optimal goodness.

The described method bears some resemblance to the approach in \cite{AnsMur16}, which also uses short-time controls to achieve real-time (nearly) optimal behavior. However, the dynamics in \cite{AnsMur16} are presumed to be known. Thus, there is no learning element to the control strategy. Additionally, the control objective is to solve a standard optimal control problem, instead of the myopic version naturally imposed by the requirements of our scenario.

The above description of the control strategy is formally given as Algorithm \ref{alg}.

\begin{table}[ht]
	\renewcommand{\arraystretch}{1.4}
	\centering
	\begin{tabular}{p{0.95\columnwidth}}
		\hline
		\vspace{-10pt}
		\begin{algorithm} 
			\label{alg} 
			\
		\vspace{-7pt}\end{algorithm} \\
		\hline
		\vspace{-5pt}
		$1$ \quad $u^*:=0$ \\
		$2$ \quad $t_0:=0$ \\
		$3$ \quad \textbf{repeat} \\
		$4$ \quad \qquad Let $x^0$ be the current state of the system. \\
		$5$ \quad \qquad Take $\Delta u^0=0$ and choose $\Delta u^1=\pm \delta e_1,\ldots,\Delta u^m=\pm \delta e_m$ such that $u^*+\Delta u^1,\ldots,u^*+\Delta u^m\in\cU$. \\
		$6$ \quad	\qquad \textbf{for} $j=0,\ldots,m$ \\
		$7$ \quad		\qquad \qquad Apply control $u^*+\Delta u^j$ in the interval of time $[t_0+j\varepsilon,t_0+(j+1)\varepsilon)$. \\
        $8$ \quad		\qquad \qquad Let $x^{j+1}$ be the system state at the end of that period. \\
        $9$ \quad	\qquad \textbf{end for} \\
		$10$ \quad	\qquad $x:=x^{m+1}$ \\	
		$11$ \quad	\qquad Define function $\tilde{v}:\cU\to\RR^n$ as follows: \\
		\qquad	\qquad \qquad If $\lambda_0,\ldots,\lambda_m$ are unique coefficients with $\sum\lambda_j=1$ such that $u=\sum\lambda_j(u^*+\Delta u^j)$, let $\tilde{v}(u)=\sum\lambda_j(x^{j+1}-x^j)/\varepsilon$. \\
        
		$12$ \quad \qquad Take $u^*\in\argmax G(\phi|_{[0,t_0+(m+1)\varepsilon]},\tilde{v}(u))$ \\
		$13$ \quad \qquad $t_0:=t_0+(m+1)\varepsilon$ \\
		$14$ \quad \textbf{until} the end of system run \\
		\hline
	\end{tabular}
\end{table}

We note that while Algorithm \ref{alg} prescribes a particular form for $\Delta^i$ in line 5, the local dynamics of the system can be learned by using any $\Delta^i$ such that $u^*+\Delta u^1,\ldots,u^*+\Delta u^m\in\cU$ is an affinely independent set. Algorithm \ref{alg} includes a particular form for $\Delta^i$ in order to more easily establish a bound on the degree of suboptimality of the resulting control law. Such a bound, given in the following section, is dependent on parameters $0<\delta\leq 1$ and $\varepsilon>0$. With an appropriate choice of $\delta$ and $\varepsilon$, we can ensure that Algorithm \ref{alg} produces a control law which comes arbitrarily close to solving Problem \ref{pro2}. This choice of $\delta$ and $\varepsilon$ will naturally depend on the bounds $M_0$ and $M_1$ on the norms and Lipschitz constants of $f$ and $g_i$, and the regularity of goodness function $G$.

\begin{remark}
\label{twoalgs}
Algorithm \ref{alg} does not distinguish between the learning and control phases of the algorithm: the algorithm learns the local dynamics as a result of performing slight perturbations of the previously established optimal control law. However, we note that these two phases can be decoupled by a minor modification of the above algorithm. After learning the local dynamics through consecutively applying any $m+1$ affinely independent controls in a short time period $\varepsilon'<\varepsilon$, the system can then apply the optimal control law derived from these dynamics for the remaining $\varepsilon-\varepsilon'$ time in the cycle, after which it begins a new cycle by learning the new local dynamics. While this modification results in a short $\varepsilon'$ period of time inside every iteration of the learn-control cycle without any guarantees on the degree of myopic suboptimality, it offers computational benefits in the case where the optimal control law does not rapidly change over time.

In light of this modification, we note that the performance bounds presented in the subsequent section apply to Algorithm \ref{alg} as originally presented. Analogous proofs could be constructed for the modified algorithm as above. For computational reasons, the algorithm used in simulations in Section \ref{damair} contains the modification described above, while the work in Section \ref{damosc} uses Algorithm \ref{alg} directly. 
\end{remark}

\section{Performance Bounds}
\label{resb}

\subsection{Learning}
\label{learph}

We claim that the procedure in lines 5--11 of Algorithm \ref{alg} produces a good approximation $\tilde{v}:\cU\to\RR^n$ of the map $v_x:\cU\to\RR^n$, with $x=x^{m+1}$ defined as in the algorithm.

We first note that, for every point $u^*\in\cU=[-1,1]^m$, if $u^*_i\geq 0$, then $-1\leq u^*_i-\delta<1$, and if $u^*_i<0$, then $-1<u^*_i+\delta<1$. Hence, for all $u^*\in\cU$ we can indeed choose $\Delta u^i=\delta e_i$ or $\Delta u^i=-\delta e_i$  such that $u^*+\Delta u^i\in\cU$, as stipulated by line 5 of Algorithm \ref{alg}. We also note that $u^*+\Delta u^i$, $i\in\{0,\ldots,m\}$, trivially form an affinely independent set.

We first describe the remainder of the procedure informally. We denote $u^*+\Delta u^i$ by $u^i$.

For each $j\in\{0,1,\ldots,m\}$, vector 
\begin{equation}
\label{appr1a}
\frac{x^{j+1}-x^j}{\varepsilon}=\frac{\phi_{u^j}(\varepsilon,x^j)-\phi_{u^j}(0,x^j)}{\varepsilon}
\end{equation} approximates the value $d\phi_{u^j}(\varepsilon,x^j)/dt=v_{x^{j+1}}(u^j)$, where we recall that $v_y(u)$ is defined by $v_y(u)=f(y)+\sum g_i(y)u_i$. Additionally, since $x^{j+1}$ and $x=x^{m+1}$ are not far apart (because $x^{m+1}$ is the state of the trajectory just $(m-j)\varepsilon$ later than $x^{j+1}$), $v_{x^{j+1}}(u^j)\approx v_{x^{m+1}}(u^j)=v_x(u^j)$.

Finally, since $u^0,\ldots,u^m$ are affinely independent, for every $u$ there exist unique $\lambda_0,\ldots,\lambda_m\in\RR$ such that $u=\lambda_0u^0+\ldots+\lambda_mu^m$ and $\lambda_0+\ldots+\lambda_m=1$. Thus, $$
f(x)+\sum_{i=1}^m g_i(x)u_i=\sum_{j=0}^m\lambda_j\left(f(x)+\sum_{i=1}^m g_i(x)u^j_i\right)\textrm{,}$$
i.e. $v_{x}(u)=\sum \lambda_jv_{x}(u^j)$. 
Since we already have approximations for $v_x(u^j)$ from \eqref{appr1a} and the subsequent discussion, we can thus approximate $v_x(u)$ by taking 
\begin{equation}
\label{appr3}
v_x(u)\approx \sum_{j=0}^m\lambda_j \frac{x^{j+1}-x^j}{\varepsilon}\textrm{.}
\end{equation}

It can be formally shown that approximation \eqref{appr3} produces an error no worse than linear in $\varepsilon$. This result is given in Theorem \ref{learapr}, while foregoing technical claims are given in Lemma \ref{lem1}.

\begin{lemma}
\label{lem1}
Let $u$ be the control law produced in one \textbf{repeat} loop of Algorithm \ref{alg} (i.e., $u$ is the concatenation of $u^0,u^1,\ldots,u^m$). Also, let $x^j$ be defined as in Algorithm \ref{alg}, where all $x^j$ belong to the same iteration of the \textbf{repeat} loop. Then, the following holds:
\begin{enumerate}[(i)]
\item For all $t_1,t_2\in[0,(m+1)\varepsilon)$, $$\|\phi_u(t_1,x^0)-\phi_u(t_2,x^0)\|\leq M_0(m+1)|t_1-t_2|\textrm{.}$$ In particular, $\|x^j-x^k\|\leq M_0(m+1)|j-k|\varepsilon$ for all $j,k\in\{0,\ldots,m+1\}$.
\item For all $j\in\{0,\ldots,m\}$, \begin{equation*}
\Bigg\|\frac{x^{j+1}-x^j}{\varepsilon}-  v_{x^{j+1}}(u^j)\Bigg\|\leq \frac{M_0M_1(m+1)^2\varepsilon}{2}\textrm{.}
\end{equation*}
\item For all $j\in\{0,\ldots,m\}$, \begin{equation*}
\left\|v_{x^{j+1}}(u^j)-v_x(u^j)\right\|\leq M_0M_1(m+1)^3\varepsilon\textrm{.}
\end{equation*}
\end{enumerate}
\end{lemma}

\begin{theorem}
	\label{learapr}
	Let $x^0,\ldots,x^{m+1}=x$ be as in Algorithm \ref{alg}, and let $u\in\cU$. Let $\lambda_0,\ldots,\lambda_m\in\RR$ be such that $u=\lambda_0u^0+\ldots+\lambda_mu^m$ and $\lambda_0+\ldots+\lambda_m=1$. Then, \begin{equation*}
\Bigg\|v_x(u)-\sum_{j=0}^m\lambda_j \frac{x_{j+1}-x_j}{\varepsilon}\Bigg\|\leq 2M_0M_1\left(m+1\right)^3\varepsilon\frac{4m^{\frac{3}{2}}+\delta}{\delta}\textrm{.}
\end{equation*}
\end{theorem}

For the sake of readability, the proofs of Lemma \ref{lem1} and Theorem \ref{learapr} are deferred to Appendix \ref{proofs}. From these two results, we conclude that the procedure in lines 5--11 indeed obtains reasonable approximations of the dynamics \eqref{thesys} at point $x=x^{m+1}$. In the next part, we propose how to use these learned dynamics to develop a control law $u^*$ which approximately solves Problem \ref{pro2}.

\subsection{Control}
\label{seccon}

From Theorem \ref{learapr} it follows that Algorithm \ref{alg} can calculate $v_{x_{m+1}}(u)=f(x_{m+1})+\sum_{i=1}^m g_i(x_{m+1})u_i$ for any value $u\in\cU$ with arbitrary precision. Thus, we are able to accurately calculate $G(\phi|_{[0,t_0+m+1\varepsilon]},v_{x^{m+1}}(u))$ for any control value $u$. Additionally, for a fixed $x$ and $\phi$, $u\mapsto G(\phi,v_x(u))$ is a real function of a bounded variable $u\in\cU$. Depending on the properties of $G$, $\max$ and $\argmax$ of this function can be hence found by an analytic or numerical optimization method; we briefly discussions numerical considerations of the solvability of this optimization problem in Section \ref{secfut}. Thus, the above approximations notwithstanding, we can find an optimal control input $u^*$ to be applied when the previous system trajectory is given by $\phi$, and the trajectory is at $x$.

As shown in Algorithm \ref{alg}, our plan is to use the found optimal control input $u^*$ for a short time $(m+1)\varepsilon$ after the system is at point $x^{m+1}$. During that time, the system will learn new local dynamics around the new $x^{m+1}$, and the whole procedure will be repeated again. We note that this process does not necessarily generate an optimal control law, as the best control input $u^*$ for a certain point $x^{m+1}$ may no longer be the best immediately after the system leaves the said point. However, if the function $G$ is ``tame enough'', $u^*$ will still be a good approximation of the optimal control input. We will show that this fact remains true even if we take into account that the input $u^*$ was calculated based on a slightly incorrect learned dynamics at $x^{m+1}$. We first introduce a measure of tameness of $G$.

\begin{definition}
	\label{deflip}
	Function $G:\cF\times\RR^n\to\RR$ has {\em Lipschitz constant} $L$ if for all $\phi_1|_{[0,T_1]},\phi_2|_{[0,T_2]}\in\cF$ and $v_1,v_2\in\RR^n$, the following holds:
	\begin{equation}
    \label{eqlip}
    \big|G (\phi_1|_{[0,T_1]},v_1)-G(\phi_2|_{[0,T_2]},v_2)\big|
    \leq L\left(d\left(\phi_1|_{[0,T_1]},\phi_2|_{[0,T_2]}\right)+\|v_1-v_2\|\right)\textrm{,}
    \end{equation}
	where 
	\begin{equation}
    \label{defd}
    d \left(\phi_1|_{[0,T_1]},\phi_2|_{[0,T_2]}\right) =|T_1-T_2|+\max_{t\in[0,\min(T_1,T_2)]}\|\phi_1(t)-\phi_2(t)\|\textrm{.}
	\end{equation}
\end{definition}

Definition \ref{deflip} explicitly makes note of the usually standard notion of a Lipschitz constant, because $d$ as defined in \eqref{defd} is not a proper metric on $\cF$. Namely, it does not satisfy the triangle inequality. Thus, the existence of a Lipschitz constant as in Definition \ref{deflip} does not imply continuity of $G$ in the standard sense.

The claims of the beginning of this section are now formalized by the following result, the proof of which is in Appendix \ref{proofs}.

\begin{theorem}
	\label{maith}
	Let $\phi_1|_{[0,T_1]},\phi_2|_{[0,T_2]}\in\cF$, and let $\nu>0$. Define $x=\phi_1(T_1)$ and $y=\phi_2(T_2)$. Assume that $G$ has Lipschitz constant $L$ in the sense of Definition \ref{deflip}, and let $u^*$ be the optimal control at $x$ under {\em approximately} learned local dynamics with the previous trajectory $\phi_1|_{[0,T_1]}$, i.e., \begin{equation*}
G \left(\phi_1|_{[0,T_1]},\tilde{f}+\sum_{i=1}^m \tilde{g}_iu^*_i\right) =\max_{u\in\cU} G\left(\phi_1|_{[0,T_1]},\tilde{f}+\sum_{i=1}^m \tilde{g}_iu_i\right)\textrm{,}
\end{equation*}
where $\|\tilde{f}+\sum_{i=1}^m \tilde{g}_iu_i-(f(x)+\sum_{i=1}^m g_i(x)u_i)\|\leq\nu$ for all $u\in\cU$. Then, 	
	\begin{equation*}
	\begin{split}
 &\left |\max_u G \left(\phi_2|_{[0,T_2]},v_y(u)\right) -G\left(\phi_2|_{[0,T_2]},v_y(u^*)\right)\right| \\
	& \textrm{\hskip 10pt}\leq 2Ld\left(\phi_1|_{[0,T_1]},\phi_2|_{[0,T_2]}\right)+2LM_1(m+1)\|x-y\| +2L\nu\textrm{.}
	\end{split}
	\end{equation*}
\end{theorem}

\begin{remark}
We note that the claim of Theorem \ref{maith}, as well as all the results that will be presented in the remainder of this paper, can be easily adapted to hold under a relaxed assumption that $G$ has a {\em local} Lipschitz constant $L$, i.e., that \eqref{eqlip} is only guaranteed to hold when $d\left(\phi_1|_{[0,T_1]},\phi_2|_{[0,T_2]}\right)<\psi$ for some $\psi>0$. However, this relaxation brings more technicality into the obtained bounds, and we hence omit further discussion on this topic.
\end{remark}

In the context of Algorithm \ref{alg}, the result of Theorem \ref{maith} can be interpreted in the following way: let us take $T_1$ to be the time at the beginning of one \textbf{repeat} loop in the algorithm, $\phi_1|_{[0,T_1]}$ to be the trajectory of the system until time $T_1$, $T_2$ to be any time in the interval $[T_1,T_1+(m+1)\varepsilon)$, and $\phi_2|_{[0,T_2]}$ to be the trajectory of the system until time $T_2$. Then, Theorem \ref{maith} essentially proves that Algorithm \ref{alg} results in an almost optimal policy with respect to \eqref{maxg}. There is only one major point that remains to be discussed. In Algorithm \ref{alg}, we do not apply {\em just} $u^*$ as the control. In order to facilitate learning, we modify this $u^*$ by some small $\Delta u^i$. We need to show that such a small control perturbation will not significantly impact the degree of suboptimality of Algorithm \ref{alg}. Corollary \ref{cor2} (with the proof in Appendix \ref{proofs}) deals with this issue.

\begin{corollary}
	\label{cor2}
	Assume the same notation as in Theorem \ref{maith}. Let $\delta>0$ and let $\tilde{u}\in\RR^m$, $\|\tilde{u}\|\leq\delta$. Then, 	
	\begin{equation*}
	\begin{split}
&	 \left|\max_u G  \left(\phi_2|_{[0,T_2]},v_y(u)\right) -G\left(\phi_2|_{[0,T_2]},v_y(u^*+\tilde{u})\right)\right|\\
	& \textrm{\hskip 10pt} \leq 2Ld\left(\phi_1|_{[0,T_1]},\phi_2|_{[0,T_2]}\right)+2LM_1(m+1)\|x-y\|+2L\nu+LM_0(m+1)\delta\textrm{.}
	\end{split}
	\end{equation*}
\end{corollary}

Finally, the following result translates the bounds in Theorem \ref{maith} and Corollary \ref{cor2} into a bound in terms of algorithm parameters $\varepsilon$ and $\delta$. Thus, Theorem \ref{bigthm}, with the proof in Appendix \ref{proofs}, finally provides a bound on the degree of suboptimality of Algorithm \ref{alg}. For notational purposes, we assume that the run of system \eqref{thesys} is of infinite length. However, the same result clearly holds for any finite length.

\begin{theorem}
	\label{bigthm}
	Let $x_0\in\cX$. Let $u^+:[0,+\infty)\to\cU$ be the control law used in Algorithm \ref{alg}. Assume that $G$ has Lipschitz constant $L$. Then, for all $t\geq (m+1)\varepsilon$, 
	\begin{equation}
	\label{subg}
	\begin{split}
	\Bigg|G & \left(\phi_{u^+}(\cdot,x_0)|_{[0,t]},\frac{d\phi_{u^+}(0+,x)}{dt}\right) -\max_{u\in\cU} G\left(\phi_{u^+}(\cdot,x_0)|_{[0,t]},\frac{d\phi_u(0+,x)}{dt}\right)\Bigg | \\
	& \textrm{\hskip 10pt} \leq 6L(M_0+1)(M_1+1)(m+1)^3\left(1+\frac{4m\sqrt{m}}{\delta}\right)\varepsilon+LM_0(m+1)\delta\textrm{,}
	\end{split}
	\end{equation}
	where $x=\phi_{u^+}(t,x_0)$.
\end{theorem}

Theorem \ref{bigthm} is the central result of the theoretical discussions of this paper. It shows that, for any $\eta>0$, if $\varepsilon$ and $\delta$ are chosen so that $6L(M_0+1)(M_1+1)(m+1)^3(1+4m\sqrt{m}/\delta)\varepsilon+LM_0(m+1)\delta\leq\eta$, Algorithm \ref{alg} will result in a control law that approximately satisfies \eqref{maxg} in Problem \ref{pro2}, with an error no larger than $\eta$. As can be seen from the proofs, the bound in Theorem \ref{bigthm} is very rough and can certainly be improved. However, an improvement on this bound is not in the focus of this paper.

Making use of the bound in Theorem \ref{bigthm} requires two elements: the controller's a priori knowledge of constants $M_0$, $M_1$, and $L$, and the ability to make $\varepsilon>0$ and $\delta>0$ as small as possible. It is difficult to escape the first requirement, although one could potentially think of a heuristic approach to learning $M_0$, $M_1$, and $L$ on the go. As for the second issue, in order to achieve arbitrarily good performance bounds on Algorithm \ref{alg}, the control actuators must allow for an arbitrarily good time resolution (i.e., shortest amount of time $\varepsilon$ that a control can be applied for) and actuation resolution (i.e., smallest amount of control $\delta$ that can be applied). While minimal $\varepsilon$ and $\delta$ that can be applied will depend on the actuator type, improvement on both of these resolutions is a topic of significant research in the engineering community \cite{Ecketal14,Krejetal17}. 

Theorem \ref{bigthm} completes this section. Before proceeding to a presentation of simulation results and a discussion of some open questions, let us briefly discuss how Algorithm \ref{alg} responds to disturbances.

\section{Robustness to Disturbances}
\label{nois}

As mentioned in the introduction, a beneficial consequence of the fact that Algorithm \ref{alg} learns the system dynamics during the system run is that it is able to easily account for the presence of disturbances. Simply put, with a trivial modification to the algorithm, the disturbance can be learned the same way as the ``intended'' underlying dynamics. In fact, the algorithm will not even know which part of the dynamics was created by the disturbance. Let us formalize the above discussion.

Consider a disturbance function $N$ which includes both time-varying and state-varying elements, i.e., $N:[0,+\infty)\times\cX\to\RR^n$. We assume that $N$ is a Lipschitz continuous function with $\|N(x)\|\leq M_0^*$ and $\|N(x)-N(y)\|\leq M_1^*\|x-y\|$ for all $x,y\in\RR^n$. After including this disturbance, dynamics \eqref{thesys} are now replaced by 
\begin{equation}
\label{thesysnoi}
\dot{x}=f(x)+\sum_{i=1}^m g_i(x)u_i+N(t,x)\textrm{.}
\end{equation}
We assume that dynamics given by \eqref{thesysnoi} still evolve on $\cX$. Because of the time-varying element of $N$, \eqref{thesysnoi} does not automatically fall within the class of systems in \eqref{thesys}. However, by appending an additional variable $x_{n+1}$ given by the dynamics $dx_{n+1}/dt=1$ and $x_{n+1}(0)=0$, and with an obvious slight abuse of notation, we obtain dynamics
\begin{equation}
\label{thesysnoi2}
\dot{x}=\begin{bmatrix} 
f(x)+N(x) \\
1
\end{bmatrix}+\sum_{i=1}^m\begin{bmatrix} 
g_i(x) \\
0
\end{bmatrix}u_i\textrm{.}
\end{equation}
System \eqref{thesysnoi2} falls into the control-affine category \eqref{thesys}, and we can thus directly apply Algorithm \ref{alg}. The only changes are in the bounds $M_0$, $M_1$: instead of $\|f\|,\|g_i\|\leq M_0$, it now holds that $$\left\|\begin{bmatrix} 
f+N \\
1
\end{bmatrix}\right\|,\left\|\begin{bmatrix} 
g_i \\
0
\end{bmatrix}\right\|\leq\max(M_0+M_0^*,1)\textrm{,}$$
and, analogously, the Lipschitz constants of the new functions in \eqref{thesysnoi2} are bounded by $M_1+M_1^*$.
Thus, if we possess a priori bounds on $M_0$, $M_1$, $M_0^*$, and $M_1^*$, $\varepsilon$ and $\delta$ in Algorithm \ref{alg} will need to be made smaller in order to guarantee the same degree of suboptimality as in the case without disturbances. With that small caveat, Algorithm \ref{alg} performs the same in the presence of disturbances as it does without it. In addition, we note that the same above explanation proves the correctness of Algorithm \ref{alg} in the generalized case of non-autonomous control-affine systems
$$\dot{x}(t)=f(t,x)+\sum_{i=1}^mg_i(t,x)u_i\textrm{,}$$
if maps $f$ and $g_i$ are Lipschitz continuous both with respect to time and position.

\section{Simulation Results}
\label{simul}

\subsection{Damaged Aircraft Example}
\label{damair}

The simulation work presented in this section is a version of the running example of this paper: a damaged aircraft that is attempting to avoid a crash. We base our model on the linearized dynamics of an undamaged Boeing 747 in landing configuration at sea level \cite{Bry94,CheChe01}. In order to model aircraft damage and introduce nonlinearity into the model, we modify the standard dynamics into the following system:

\begin{equation}
\label{dam}
\begin{split}
\dot{x}_1= & -0.021x_1+0.122x_2-0.322x_3+0.01u_1+u_2 \\
\dot{x}_2= & -0.209x_1-0.53x_2+2.21x_3-0.064u_1-0.044u_2 \\
\dot{x}_3= & \textrm{ } 0.017x_1+0.01\cos(x_1)x_1-0.164x_2+0.15\sin(x_1)x_2 \\
& -0.421x_3-0.378u_1+0.544u_2+0.5\sin(x_2)u_2 \\
\dot{x}_4= & \textrm{ } x_3 \\
\dot{x}_5= & -x_2+2.21x_4\textrm{,}
\end{split}
\end{equation}
where $x=(w_l,w_v,q,\theta,h)$ and $u=(\delta_e,\delta_t)$. Variable $w_l$ represents the deviation from the steady-state longitudinal speed of the aircraft, $w_v$ the downwards-pointing vertical speed, $q$ the pitch rate, $\theta$ the pitch angle and $h$ the altitude. Additionally, inputs $\delta_e$ and $\delta_t$ are the deviations from the elevator and thrust values used for level flight. All values are in feet, seconds and centiradians. The limits on control inputs are set to $u_1=\delta_e\in[-30,30]$ and $u_2=\delta_t\in[-1,1]$, roughly informed by the descriptions in \cite{Ganetal02,Smi75}.


System \eqref{dam} gives the equations of motions that the simulated aircraft will follow. The only difference between \eqref{dam} and the model of an undamaged aircraft from \cite{Bry94,CheChe01} is in the presence of nonlinear factors, which do not appear in \cite{Bry94,CheChe01}. While we recognize that \eqref{dam} is not necessarily realistic for any particular type of aircraft damage, and the introduction of trigonometric terms was partly motivated by the desire to showcase how our control strategy handles nonlinearities, dynamics \eqref{dam} were generally chosen in order to emulate damage to the aircraft's horizontal stabilizer, which results in the pitch rate behaving more erratically. We also remind the reader that our controller design does not start with any a priori knowledge of the system dynamics, and our controller will thus not be making use of any similarities between the undamaged dynamics from \cite{Bry94,CheChe01} and actual dynamics \eqref{dam}.

As mentioned throughout the previous sections, the setting investigated in this example is that the aircraft suffered damaged while in flight, and its main objective is to remain in the air. We now make this scenario more precise: the initial state (i.e., the state at which the abrupt in change in system dynamics occured) is $(0,0,0,0,100)$, that is, the aircraft was in level flight at an altitude of $100$ feet. The primary goal of the aircraft is:
\begin{enumerate}
	\item[(i)] $x_5(t)>0$ for all $t\geq 0$.
\end{enumerate}
We also impose further specifications:
\begin{enumerate}
	\item[(ii.a)] $x_4(t)\in[-40,40]$ for all $t\geq 0$  (high pitch angles are unsafe for the aircraft; the bound is informed by \cite{MulSte93}),
	\item[(ii.b)] $\dot{x}_5(t)\in[-30,30]$ for all $t\geq 0$ (high climb rates are unsafe for the aircraft; the bound is informed by \cite{MulSte93}), and
	\item[(iii)] $x_5(t)\in[900,1100]$ for all $t\geq T$, with $T>0$ as small as possible (because of the damage and the fact that the aircraft is in landing configuration, it may not be either possible nor smart to move to a high altitude, but it is also not desirable to stay very close to the ground).
\end{enumerate}

While we want to ensure that all above restrictions are satisfied, it is possible that at certain times, the above specifications will clash with each other. We impose the following priority, ranked in descending order: (i), (ii.a) and (ii.b), (iii).

We now design a goodness function corresponding to the above specifications. As we described in Section \ref{measgoo}, there is no formally correct way for developing a goodness function. Its design depends primarily on the geometric intuition of the problem. Hence, we first give a rough idea behind the goodness function that we wish to develop:

\vskip 10pt
	\begin{tabular}{p{0.38\columnwidth}p{0.62\columnwidth}}
		if $h=x_5<100$: & $G$ is such that $x_5$ quickly increases, \\
        else if $|x_4|>40$ or $|\dot{x}_5|>30$: & $G$ is such that $|x_4|$, $|\dot{x}_5|$ or both (as needed) quickly decrease, \\
        else if $x_5\notin[900,1100]$: & $G$ is such that $x_5$ quickly approaches $1000$, \\
        else: & $G$ is such that $|x_4|$ and $|\dot{x}_5|$ remain small.
	\end{tabular}
\vskip 10pt

The above informal design remains to be converted into a goodness function $G$. In particular, based on the above motivation, we need to define six functions that will result in following motions: $x_5$ quickly increasing, $x_5$ quickly decreasing, $|x_4|$ quickly decreasing, $|\dot{x}_5|$ quickly decreasing, $|x_4|$ remaining around $0$, $|\dot{x}_5|$ remaining around $0$.

Defining the above functions would not be difficult if we had full knowledge of the system dynamics \eqref{dam}. However, since the setting of this example deals with an unknown abrupt change in system dynamics, these dynamics are ``hidden'' from us when designing the goodness function. Instead, the only a priori knowledge about system dynamics comes from basic physical laws and the physical meaning of $x_1,\ldots,x_5$. In particular, we know the following:
\begin{enumerate}[1$^\circ$]
\item Increasing $u_1=\delta_e$ will decrease $x_2=w_v$ and $x_3=q$ (by the design of the elevator);
\item increasing $u_2=\delta_t$ will increase $x_1=w_l$ and $x_3=q$, and decrease $x_2=w_v$ (by the function of the aircraft engine controls);
\item $\dot{x}_4=x_3$ (by definition of $x_3$ and $x_4$);
\item disregarding large deviations in the longitudinal speed, $\dot{x}_5$ depends on $x_2$ and $x_4$: increasing $x_4$ will increase $\dot{x}_5$, and increasing $x_2$ will decrease $\dot{x}_5$ (by definitions of $x_2$, $x_4$, and $x_5$, and the fact that an aircraft will always have a positive longitudinal speed); and
\item $u_1$ and $u_2$ do not directly increase or decrease $x_4$ and $x_5$, but instead act on them through $x_2$ and $x_3$.
\end{enumerate}

We once again emphasize that our strategy does not require any specific knowledge on the system dynamics \eqref{dam}: facts $1^\circ$-$5^\circ$ all follow from physical laws, definitions of the variables, and basic understanding of how aircraft control inputs work.

Let us now go back to designing functions that result in the required six motions. From $5^\circ$, we know that we cannot directly increase $x_5$ through control inputs. Thus, from $4^\circ$, in order to increase $x_5$ (and $\dot{x}_5$) we know that we want to increase $x_4$ and decrease $x_2$. Since by $5^\circ$ and $3^\circ$ we cannot directly influence $x_4$ through our inputs, but only $x_3$, we want to find a control input that results in an increase of $x_3$ and decrease of $x_2$. We do not know exactly which of these two methods we should prioritize, so it makes sense to find an input that maximizes the value $-\dot{x}_2+\dot{x}_3$. Analogously, in order to decrease $\dot{x}_5$ and $x_5$, we want to maximize $\dot{x}_2-\dot{x}_3$.

What remains is to find methods for decreasing $|x_4|$ and keeping $|x_4|$ and $|\dot{x}_5|$ around $0$. Let us first deal with the tasks pertaining to $x_4$. In order to decrease $|x_4|$, we maximize the increase of $x_3=\dot{x}_4$ in the direction opposite from $x_4$, thus providing negative acceleration which will ultimately result in $x_4$ reducing to $0$. In order for $|x_4|$ to remain around $0$, we simply use the above method to decrease $|x_4|$ whenever it grows beyond some small number (e.g., $|x_4|>2$).

By $4^\circ$, keeping $|\dot{x}_5|$ around $0$ can be accomplished by reducing $|x_4|$, for which we already designed a strategy, and reducing $|x_2|$. We use a similar method as above to reduce $|x_2|$. We first bring the values of $x_2$ to around $0$ by negative acceleration, and then re-apply this method whenever $|x_2|$ grows beyond $2$.

Using the above methods, our rough idea of a goodness function $G$ is now formalized as follows:
    \vskip 10pt
	\begin{tabular}{p{0.38\columnwidth}p{0.62\columnwidth}}
		if $x_5(t)<100$: & $G(x|_{[0,t]},v)=v_3-v_2$, \\
        else if $|x_4(t)|>40$ or if $|\dot{x}_5(t)|>30$ \newline or if $x_5(t)\in[900,1100]$: & $G(x|_{[0,t]},v)=m_1+m_2$, where $m_1=-v_2\cdot\sign (x_2(t))$ \newline if $|x_2(t)|>2$ and $m_1=0$ otherwise, and \newline $m_2=
        -v_3\cdot\sign (x_4(t))$ if $|x_4(t)|>2$, and $m_2=0$  otherwise, \\
        else: & $G(x|_{[0,t]},v)=(v_3-v_2)\cdot\sign (1000-x_5(t))$.
	\end{tabular}
    \vskip 10pt

Once again, we emphasize that this goodness function is not the only possible one. It is certainly possible to come to an equally reasonable, but different function through different logic. One possible goodness function is described in Example \ref{exrap}, which proposes a continuous goodness function for the reach-avoid problem. However, that function  depends on the ``mixing'' parameter that prioritizes the two objectives, and such a parameter would be difficult to design. On the other hand, our proposed goodness function is simpler to design, but is not continuous. Hence, we are unable to directly use the results of Section \ref{resb} to determine good choices of cycle length $\varepsilon$ and wiggle size $\delta$. Additionally, for computational reasons, instead of applying Algorithm \ref{alg} directly as written, we employ it with a minor modification as in Remark \ref{twoalgs}, where learning and control periods inside one learn-control cycle are decoupled. We chose the length of the learning period to be $\varepsilon'=10^{-4}$, and the length of the entire learn-control cycle to be $10^{-1}$.

The simulation results are presented in Figure \ref{simfig1}. From Figure \ref{simfig1}, we note that the algorithm, i.e., the myopic control strategy, performs exactly as desired. Within $\sim 75$ seconds, the aircraft increases its height to roughly $1000$ feet. After that, it takes another $\sim 75$ seconds to safely position its altitude within $[900,1100]$. From then on, the system stabilizes to an altitude of $\sim 960$ feet. Additionally, the aircraft's pitch angle $\theta$ and climb rate $\dot{h}$ almost always remain in $[-40,40]$ and $[-30,30]$, respectively. While the peaks that result in the trajectory leaving the desired bounds could be avoided by using a different goodness function in the climb period, we did not make such changes in order to emphasize that the control design performs exactly as intended: as soon as the climb rate or angle leave the desired bounds, the controller takes immediate action to return them back into the desired region.

\begin{figure}[ht!]
	\centering
	\includegraphics[height=110pt, width=220pt]{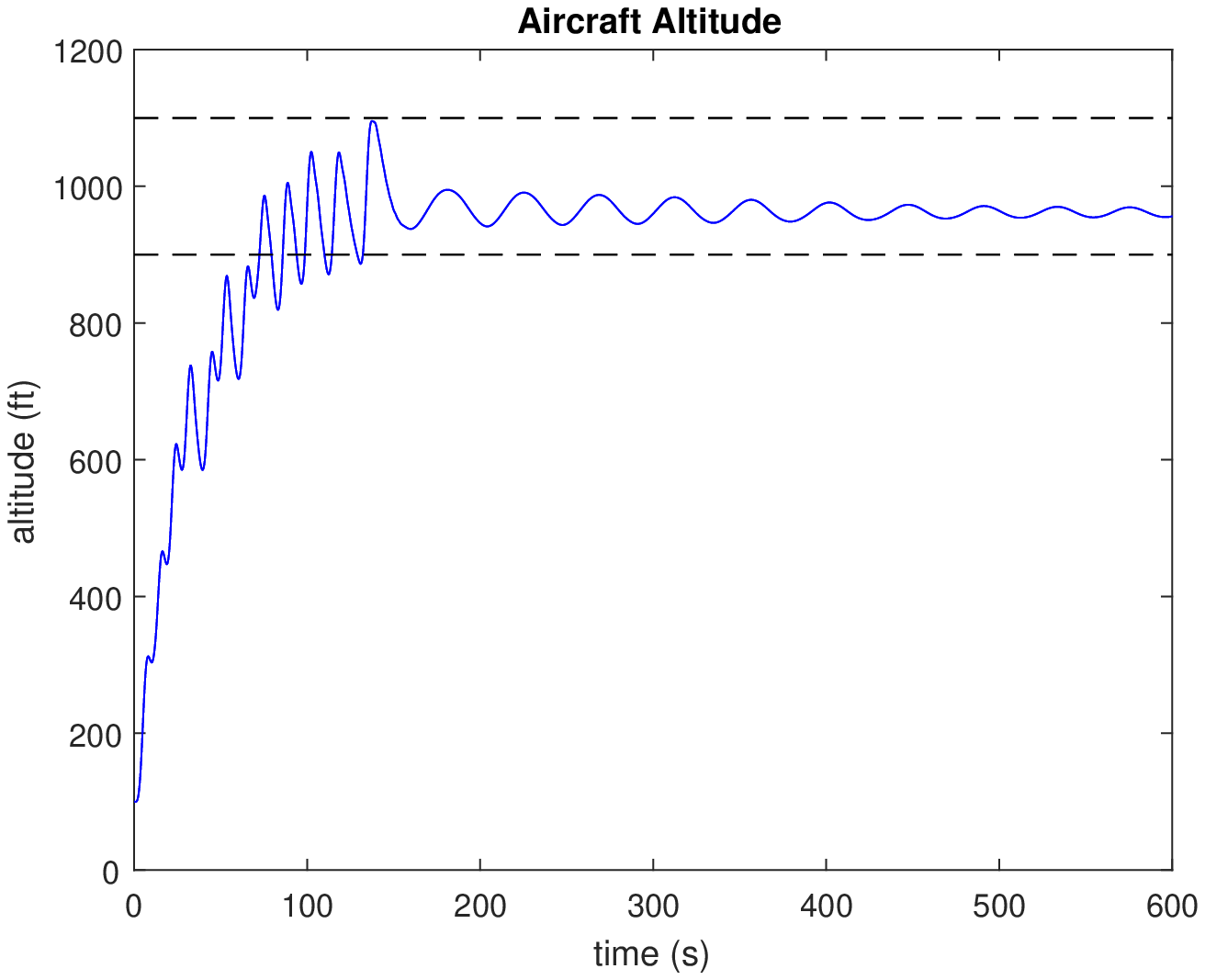}
    \includegraphics[height=110pt, width=220pt]{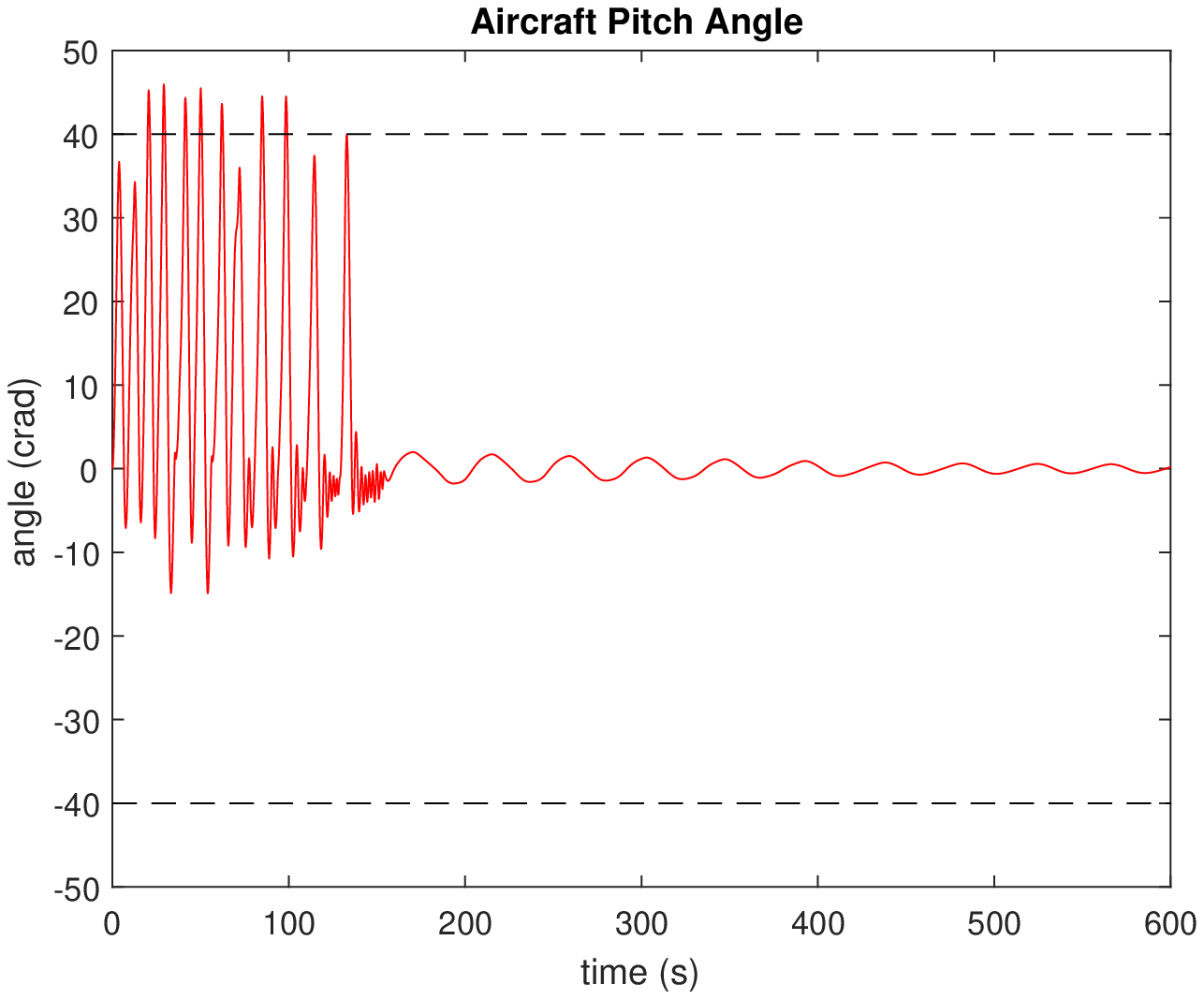}
    \includegraphics[height=110pt, width=220pt]{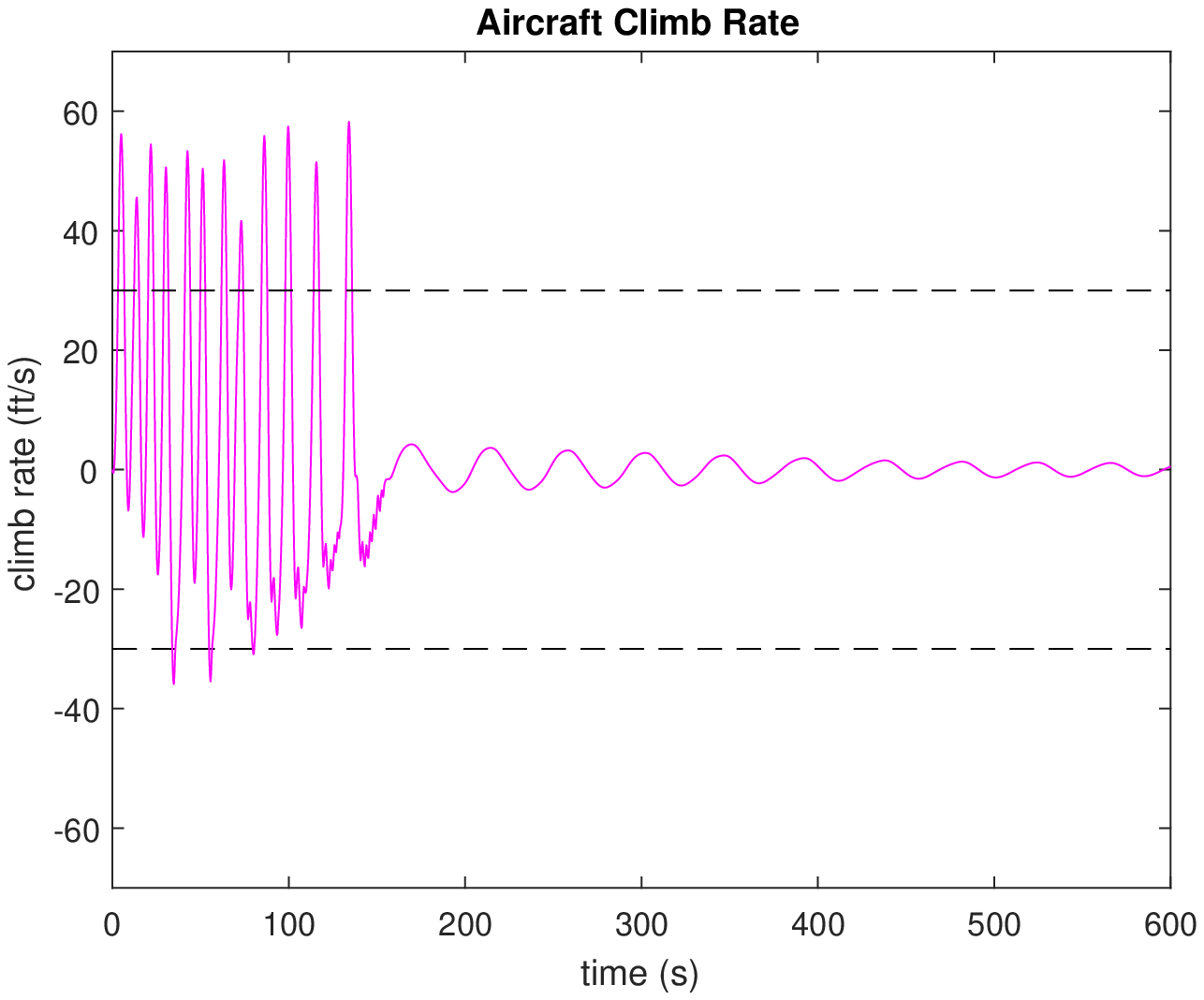}
\caption{Graphs of the simulated aircraft trajectory from the setting of Section \ref{damair}. The aircraft altitude $h$ is given in blue, the pitch angle $\theta$ is in red, and the climb rate $\dot{h}$ is in pink. The desired bounds for these values are denoted by dashed lines on the corresponding graph.}
	\label{simfig1}
\end{figure}

We note that the aircraft performs better than intended. In particular, it asymptotically stabilizes to a level flight at a particular altitude, which was neither required by the problem specifications, nor explicitly enforced by the goodness function. However, it makes natural sense given the system dynamics: when the aircraft is in an almost perfect position (i.e., $|x_2|\leq 2$, $|x_4|\leq 2$, $x_5\in[900,1100]$), the goodness function is $0$. In that case, the system does not have any preference for the given control. It thus always applies $u=0$. Then, the system dynamics are given as $\dot{x}=f(x)$, which, in the case of dynamics \eqref{dam}, results in convergence to a local equilibrium. By a smart choice of a goodness function $G$, the system could be made to asymptotically converge to any altitude. However, for this behavior to be a priori guaranteed, we would need to have significant prior knowledge of dynamics \eqref{dam}.

\subsection{Van der Pol Oscillator}
\label{damosc}

We consider the following control system:

\begin{equation}
\label{vanderpol}
\begin{pmatrix}
\dot{x}_1 \\
\dot{x}_2
\end{pmatrix}=\begin{pmatrix}
x_2 \\
-x_1 - 2(1-x_1^2)x_2
\end{pmatrix}+\begin{pmatrix}
0 \\
1
\end{pmatrix}
u\textrm{,}
\end{equation}
with the initial value $x(0)=(1,-2)$ and control input bounds $u\in[-2,2]$.

Equation \eqref{vanderpol} models a Van der Pol oscillator with control input. The Van der Pol oscillator is a standard example of a nonlinear control system (see, e.g., \cite{Vid02} for a longer discussion) and is widely used in applications, including human locomotion \cite{Dutetal03}, robotics \cite{HelEsp08, VesDem05}, and laser dynamics \cite{WirRan02}.

The control objective in this example is to bring the value of $x_2$ around $0$, and control it so that it remains as close as possible to that value. In a sense, this is a ``poor man's stabilization''. As myopic control is by its nature not well-suited for asymptotic results, we are not explicitly requiring $x_2(t)\to 0$ as $t\to\infty$.

The design of goodness function $G$ proceeds similarly to Section \ref{damair}. We are not using any particular information on system \eqref{vanderpol}, but we are assuming that it is known that $x_2$ can be directly influenced by the choice of $u$. Then, by the same logic as in Example \ref{expro1} and Section \ref{damair}, in order to minimize $|x_2|$, it makes sense for $G$ to measure how fast $x_2$ approaches $0$. As in Section \ref{damair}, this results in a goodness function $G(x|_{[0,t]},v)=-v_2\sign(x_2(t))$. However, as all our results in Section \ref{resb} require $G$ to be continuous, in this example we mollify this function and take 
\begin{equation}
\label{ourg}
G(x|_{[0,t]},v)=-v_2\arctan\left(x_2(t)\right)\textrm{.}
\end{equation}
Since for every $x_2\in\RR\backslash\{0\}$ the value of $v_2\arctan(x_2)$ is a positive scalar multiple of $v_2\sign(x_2)$, this mollification does not change $\argmax G$. The control input which maximizes $G$ will still be given by $u=2$ if $x_2<0$ and $u=-2$ if $x_2>0$. However, the penalty for choosing a slightly suboptimal control instead of the optimal one will be lower when $x_2$ is closer to $0$. This makes natural sense: the closer the value of $x_2$ is to $0$, the more satisfied with the current state we are, and we can afford to be more permissive with the direction in which the system is moving.

Before presenting the simulation results, let us briefly explore the performance bounds provided by Section \ref{resb}. Since the goodness function $G$ from \eqref{ourg} is Lipschitz-continuous, we can apply the results from Section \ref{resb}. In order to do that, we need bounds on the norm of the system dynamics, as well as Lipschitz constants of system dynamics and goodness function. While the system is technically evolving on all of $\RR^2$, we will assume that we are only interested in the dynamics that take place in $\cX=[-5,5]\times[-5,5]$. Then, it is not difficult to show that $M_0\leq 250$, $M_1\leq 99$, $L\leq 29$. By plugging these values into Theorem \ref{bigthm}, the myopic suboptimality of the control law generated by Algorithm \ref{alg} is bounded by $B_{\varepsilon\delta} \approx 3.5\cdot 10^7(1+4/\delta)\varepsilon+1.5\cdot 10^4\delta$.
We note that, for a given direction vector $v=\dot{x}(t)$, by \eqref{ourg}, $x_2$ will be moving towards $0$ if $G(x|_{[0,t]},v)>0$. Thus, as long as 
$\max_u G(x|_{[0,t]},v_{x(t)}(u))>B_{\varepsilon\delta}$ the system will be moving in the right direction. 

Since the goodness function $G$ as defined in \eqref{ourg} only depends on $x(t)\in\RR^2$ and $v=v(u)$, inequality $\max_u G(x|_{[0,t]},v_{x(t)}(u))>B_{\varepsilon\delta}$ defines a region $\cG$ of $\RR^2$ in which the system is guaranteed to be progressing in a good direction (i.e., $x_2$ is approaching $0$). However, simple calculations show that in order to guarantee that the trajectory starting from $x_0=(1,-2)$ will even start moving in the right direction (after the first $\epsilon$ period of time), it is necessary to have $\varepsilon<10^{-7}$ and $\delta<10^{-4}$. However, these bounds are derived from Section \ref{resb} and are very conservative. In fact, simulations show these value can be larger by several orders of magnitude without impacting the performance of the system. Figure \ref{simfig2} shows the simulation results with $\varepsilon=10^{-4}$ and $\delta=10^{-3}$.

\begin{figure}[ht!]
	\centering
	\includegraphics[height=110pt, width=220pt]{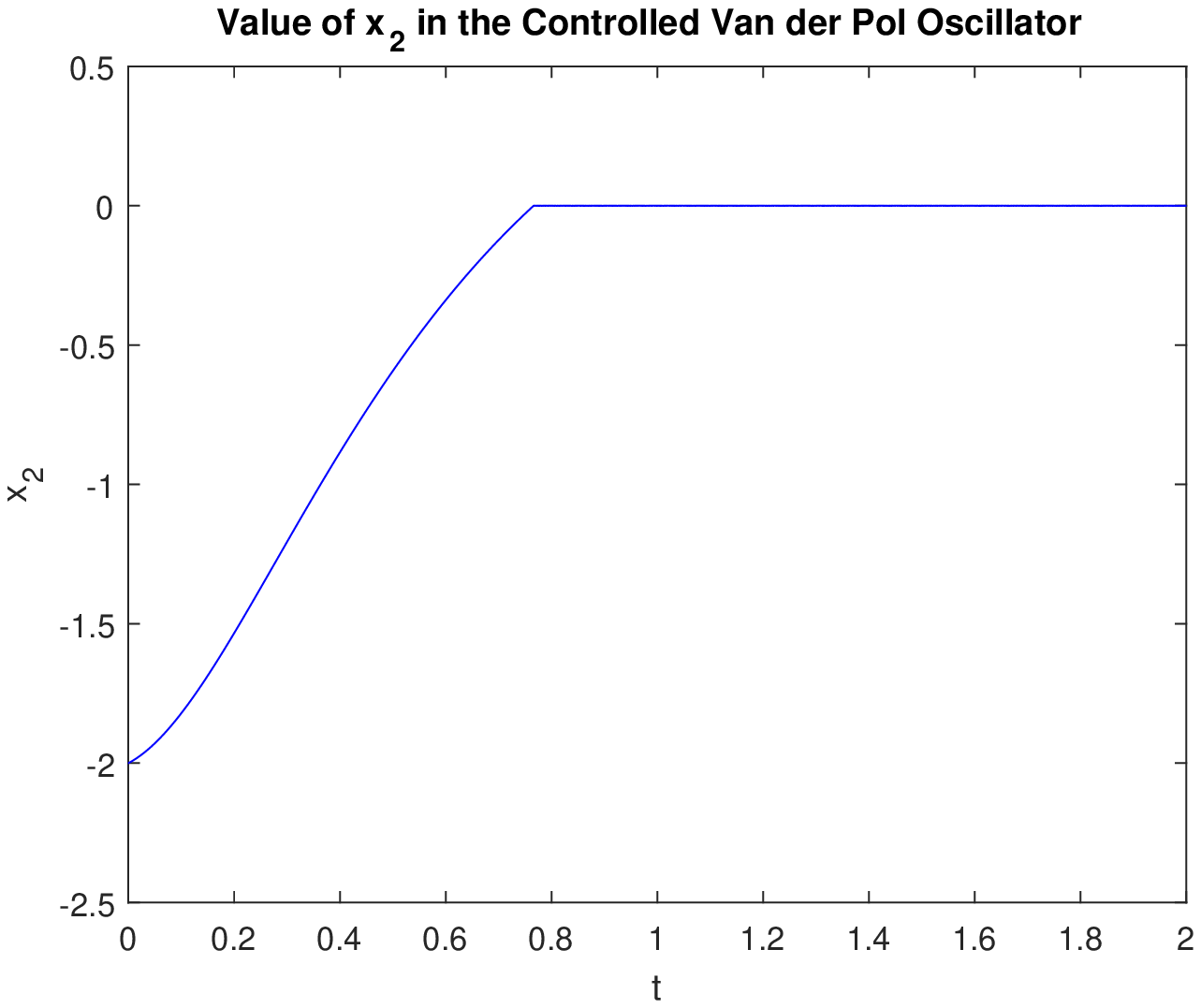}
    \includegraphics[height=110pt, width=220pt]{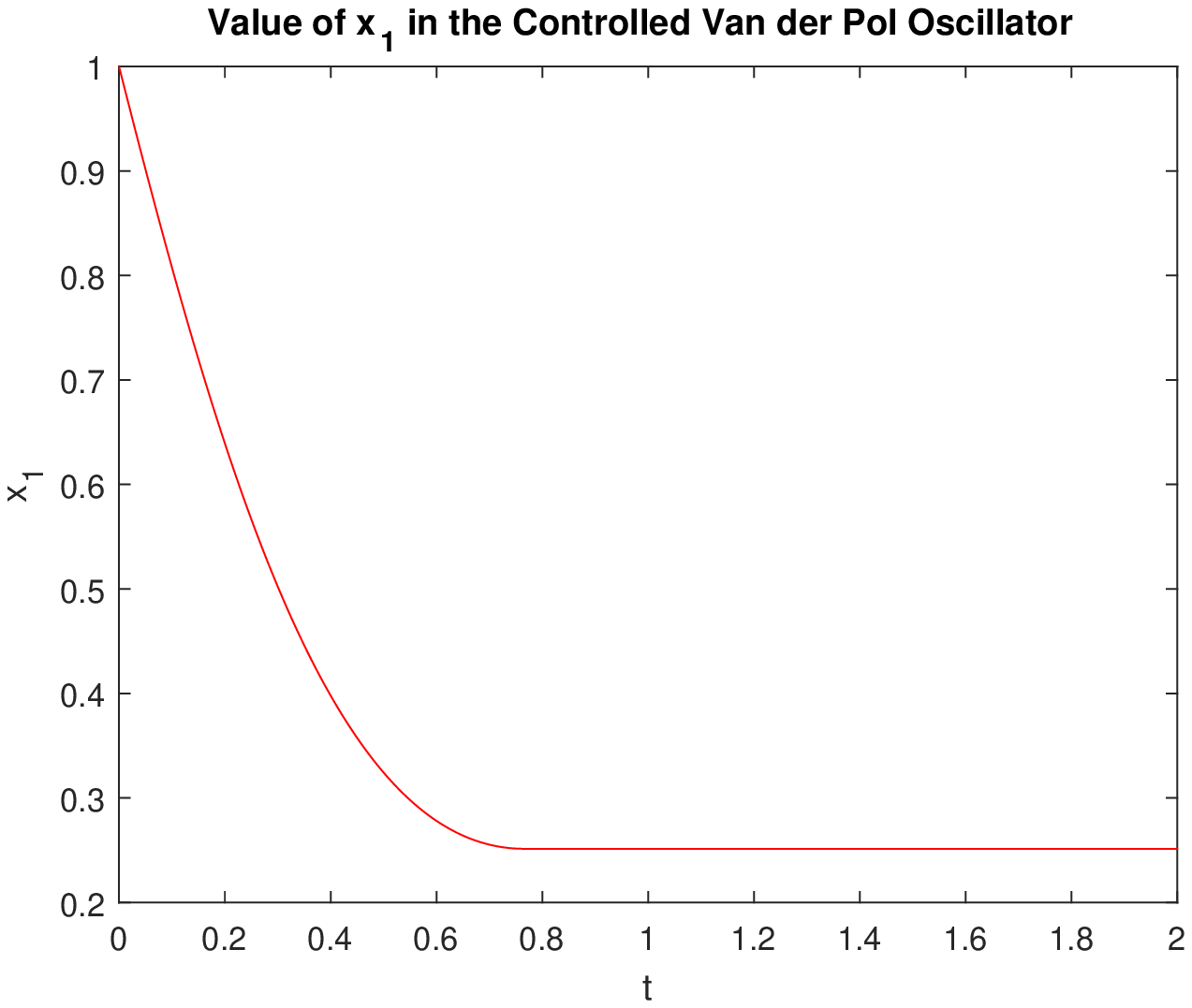}
    \includegraphics[height=110pt, width=220pt]{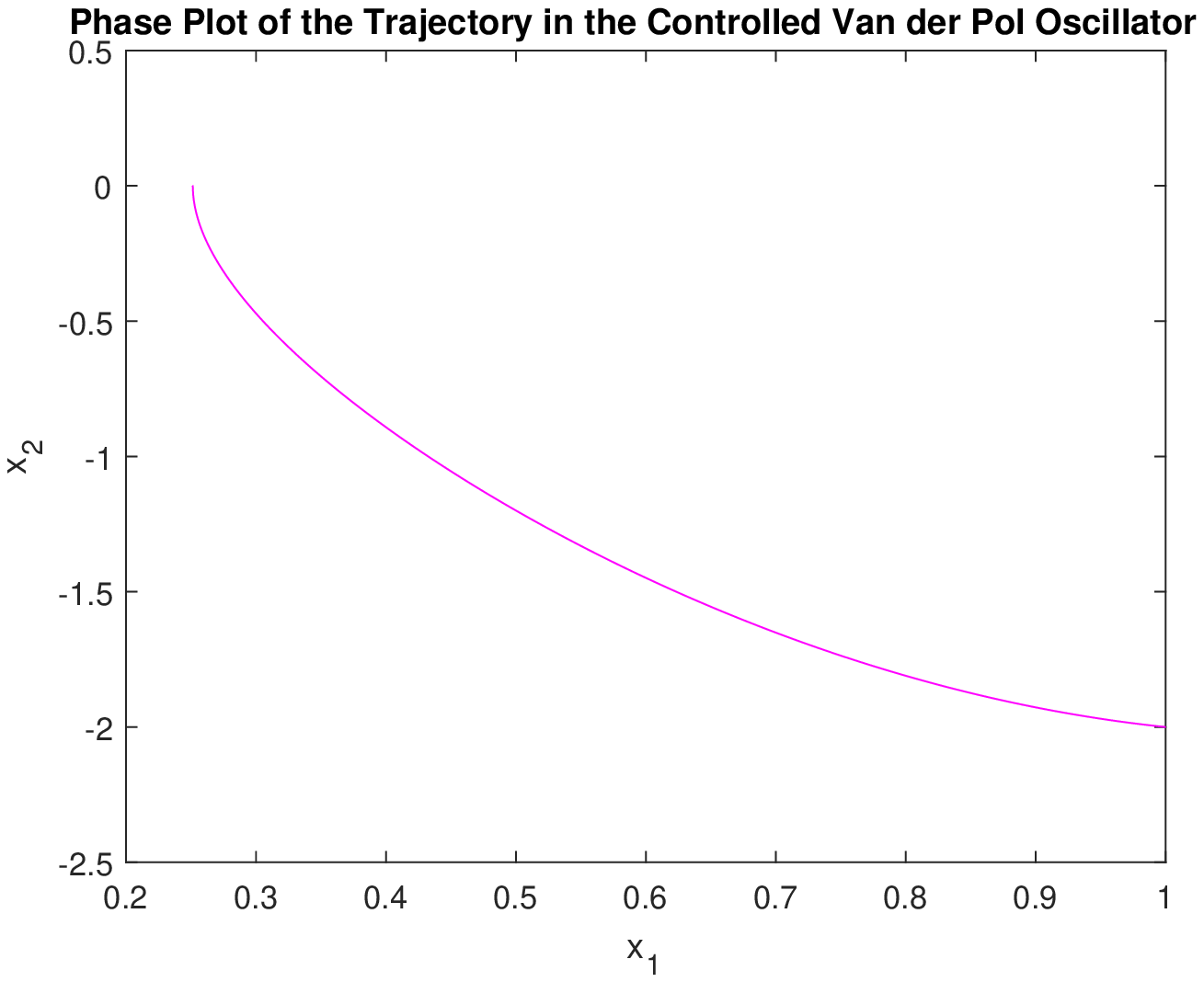}
\caption{Graphs of the simulated trajectory from the setting of Section \ref{damosc}. The blue line represents the value of $x_2$ over time, the red line represents the value of $x_1$ over time, and the pink line is the phase plot of the trajectory.}
	\label{simfig2}
\end{figure}

As shown in Figure \ref{simfig2}, the control strategy performs exactly as desired. The trajectory moves in the right direction from the beginning and reaches $x_2=0$. After reaching it, it remains in $|x_2|<5\cdot 10^{-3}$, constantly switching the controller in order to remain close to the desired point. We also note that $x_1$ also remains within $10^{-3}$ of a particular value. This was not guaranteed nor intended by the design of the goodness function, but is a natural consequence of the fact that $x_2$ remains around $0$ and keeps switching signs. Thus, $
\dot{x}_1$ also remains around $0$, and every increase in $x_1$ is followed by a decrease of roughly the same size.

We once again emphasize that the wiggle size $\delta$ and cycle time $\epsilon$ were several orders of magnitude larger than required to guarantee good behaviour by the results of Section \ref{resb}, and yet, the system works as intended. While such a result is partly a natural consequence of the liberal bounds on the norms and Lipschitz constants of the system in this example, it additionally shows that the bounds given in Section \ref{resb} provide much more conservative guarantees than what is obtained in practice.

\section{Future Work}
\label{secfut}

Before providing a concluding summary of the paper, let us list and briefly discuss some remaining open issues and possible future directions of work based on the contributions of this paper. We generally proceed in the order that the issues appear in the paper.

\begin{itemize}
	\item{The paper makes an assumption on the control-affine structure \eqref{thesys} of the control system. While control-affine systems are a wide and physically interesting class, it is of significant interest to try to generalize the approach of this paper to general nonlinear systems. A naive approach to doing so would be to separate the learning and the control phase of Algorithm \ref{alg}, along the same lines as in Remark \ref{twoalgs}. Then, in the learning phase, the algorithm would perform an exhaustive search over the space $\cU$, choose the best control, and apply it for a certain time period in the control phase. After that, the process would be repeated. However, this method would not provide good guarantees for algorithm performance: the number of tested controls in the learning phase, which contains no performance guarantees, would be enormous, and the learning would hence take so long that it could lead to significant performance degradation.
	
	A useful, and seemingly not difficult, extension would be to adapt Algorithm \ref{alg} to control systems which, while not control-affine, still have some structure which allows the algorithm to learn control dynamics around a point from only a small number of tested controls. The class of such systems, for example, includes systems polynomial in control (discussed, e.g., in \cite{MouPer05}).
	}
	\item{An interesting extension of this work would be to try to account for piecewise-continuity of $f$, $g_i$. This setting corresponds to \eqref{thesys} being a hybrid/switched system. Currently, the presented algorithm has no way to detect whether, or make use of the fact that the system made a jump into a different mode of dynamics. The addition of such an ability would potentially also enable identifying which mode of dynamics is the best for a particular control objective.
	}
	\item {
	While the case of partial observations, i.e., partial knowledge of the system trajectory, is of significant interest for physical systems \cite{FleRis75,Mar90}, it is not explored in this paper. A possible approach to that scenario is to use test controls $u^*+\Delta u^0,\ldots,u^*+\Delta u^m$ in such a way that, while each control only reveals partial information of the dynamics around a point, the information gathered from all the controls combined allows the controller to reconstruct full local dynamics.
	}
	\item {
	The primary question still left partly unexplored by this paper concerns the choice of a good function $G$. We discussed some details on this choice in Section \ref{measgoo}. However, that was merely a preliminary exploration. It would be useful to obtain results which provide a formal relationship between a solution to Problem \ref{pro2} (for a particular function $G$) and solutions to Problem \ref{rap}, Problem \ref{pro1}, and other non-myopic control problems. 
	
The addition of side knowledge into the decision process by choosing an appropriate function $G$ was briefly discussed in Section \ref{measgoo}, and information coming from physical laws was successfuly used in the design of a function $G$ in Section \ref{damair}. However, no formal work on the use of side information in goodness function design has been presented in this paper.
	
	On the side of lower-level considerations for the choice of function $G$, we note that in Algorithm \ref{alg}, we assume that $
\argmax_u G\big(\phi,\tilde{v}(u)\big)$, or at least one element of it, can be found, and that this operation can be done momentarily. However, the validity of this assumption depends on the structure of the function $G$: if, for instance, $G$ is affine in the second variable, as in Remark \ref{exrem}, the solution of the above optimization problem is trivial. If, however, $G$ is highly nonlinear, determining $\argmax_u G\big(\phi,\tilde{v}(u)\big)$ can become difficult. Thus, it might make sense to choose $G$ to be, e.g., convex, affine, or quadratic in some sense.
	
	Additionally, we note that the results of Section \ref{seccon} are all made under the assumption that $G$ has a Lipschitz constant. While this assumption often makes sense, Example \ref{exrap} introduces a goodness function which is not continuous in $\phi(t)$. If $G$ does not have a Lipschitz constant, the performance guarantees of Section \ref{seccon} would have to be reworked using some other measure of ``tameness'' of $G$ --- for instance, a possible approach may be to use locally bounded variations.
	}
\end{itemize}

\section{Conclusion}
\label{conc}

This paper proposes a method for control of an a priori unknown control-affine system in order to solve reach-avoid-type problems. Because of obvious limitations when controlling a system with unknown dynamics, real-time design of an optimal control law is impossible. Instead, we propose a notion of myopic control, where the optimal control law is one which, at a given time, seems to best satisfy the control specifications. We propose an algorithm that finds a near-optimal solution to the myopic control problem. The algorithm works by simultaneously learning the local system dynamics by applying a series of small controls and using those dynamics to find an optimal control input for a given instance of time, and can produce a control law arbitrarily close to the optimal one. Because the proposed algorithm does not use previously known system dynamics, but learns them on the way, it is highly robust to noise. We provide extensive discussion and simulation results, in particular, for a particular example of control of an aircraft that underwent a catastrophic failure, and, finally, describe future avenues of research based on the proposed algorithm.

\appendices
\section{Existence of a Myopically-Optimal Control Law}
\label{exim}

Let $\tilde\cS:\cF\to 2^{\cU}$ be a set-valued map defined by $\tilde\cS\left(\phi|_{[0,t]}\right)=\argmax_{u\in\cU} G\left(\phi|_{[0,t]},v_{\phi(t)}(u)\right)$. Then, Problem \ref{pro2} has a solution if and only if the following problem has a solution: 
\begin{equation}
\label{seteq2}
\begin{split}
& \dot{x}(t)\in \left\{f\left(x(t)\right)+\sum_{i=1}^m g_i\left(x(t)\right)u_i~|~u\in\tilde{\cS}\left(x|_{[0,t]}\right)\right\}\textrm{,} \\
& x(0)=x_0\textrm{.}
\end{split}
\end{equation}
If we define a set-valued map $S:\cF\to 2^{\RR^n}$ as the right-hand side of the differential inclusion in \eqref{seteq2}, the existence of solution to Problem \ref{pro2} can thus be interpreted in terms of the existence of a solution to \begin{equation}
\label{seteq}
\begin{split}
& \dot{x}(t)\in S(x|_{[0,t]})\textrm{,} \\
& x(0)=x_0\textrm{,}
\end{split}
\end{equation}
that is, a delay differential inclusion in the sense of \cite{Bou10}. 

Differential inclusion \eqref{seteq} differs from the setting covered by the bulk of research in delay differential inclusions (e.g., \cite{Car97, HenOua09, Liuetal16}), as the delay (i.e., the length of time interval on which the right side depends) in \eqref{seteq} is not fixed. Nonetheless, some sufficient conditions for the existence of a solution to \eqref{seteq} have been previously identified. In particular, Theorem \ref{thbou} presents an adaptation of the central result of \cite{Ana75}. As the work in \cite{Ana75} only deals with $\cX=\RR^n$, we make that assumption in the remainder of this section. In the context of Theorem \ref{thbou}, we define the norm on $C([0,t],\RR^n)$ by $\|\phi\|=\max_{\tau\in[0,t]}\|\phi(t)\|$. We additionally remind the reader of the following definitions.

\begin{definition}
Let $\cA$, $\cB$ be topological spaces, and let $f:\cA\to 2^\cB$. Map $f$ is {\em upper hemicontinuous} at a point $a\in \cA$ if for any open set $O_B\supseteq f(a)$ in $\cB$, there exists an open set $O_A\ni a$ in $A$ such that $f(a')\subseteq O_b$ for all $a'\in O_A$. Map $f$ is {\em upper hemicontinuous} if it is upper hemicontinuous at all points $a\in A$.
\end{definition}

\begin{definition}
Let $\cB$ be a topological space with a Lebesgue measure, and $I\subseteq\RR$ an interval. A map $f:I\to 2^\cB$ is {\em Lebesgue measurable} if for every closed set $C\subseteq\cB$, the set $\{t\in I~|~f(t)\cap C\neq\emptyset\}\subseteq I$ is Lebesgue measurable.
\end{definition}

We are now ready to present a sufficient condition for the existence of a solution to Problem \ref{pro2}.

\begin{theorem}
\label{thbou}
Let $\gamma>0$, and let $\cX=\RR^n$. Assume that all of the following hold:
\begin{enumerate}[(i)]
\item $S(\phi|_{[0,t]})\subseteq\RR^n$ is a convex set for all $t\in[0,\gamma]$ and all continuous maps $\phi|_{[0,t]}:[0,t]\to\cX$.
\item For any fixed $t\in[0,\gamma]$, the restriction of the map $S$ to $C([0,t],\cX)$ is upper hemicontinuous.
\item For any fixed $\phi\in C([0,\gamma],\cX)$, the map $S_\phi:[0,\gamma]\to 2^{\RR^n}$ defined by $S_\phi(t)=S(\phi_{[0,t]})$ is Lebesgue-measurable.
\item There exists $k>0$ such that, for every absolutely continuous function $\phi\in C([0,\gamma],\cX)$, every $t\in[0,\gamma]$ such that $d\phi(t)/dt$ exists, and every $y\in S(\phi|_{[0,t]})$, it holds that $y\cdot d\phi(t)/dt\leq k(1+\|\phi|_{[0,t]}\|^2)$.
\end{enumerate}
Then, \eqref{seteq} has a solution on the interval $[0,\gamma]$.
\end{theorem}

A slightly more permissive, but significantly less elegant, sufficient condition for the existence of a solution to \eqref{seteq} has been identified in \cite{Bou10}; we direct the reader to the details of that paper for more information. We forgo the proof of Theorem \ref{thbou}, as it follows immediately from the main result in \cite{Ana75}. We note that the statement of Theorem \ref{thbou} is somewhat simpler than the original theorem statement of \cite{Ana75}, as one of the original conditions is automatically satisfied because $\cU$ is assumed to be bounded.

\begin{remark} 
It can be shown that the goodness function from the example of the Van der Pol oscillator in Section \ref{damosc} does not satisfy the sufficient condition of Theorem \ref{thbou}. Nevertheless, as can be seen from Section \ref{damosc}, this does not prevent Algorithm \ref{alg} from constructing a control law arbitrarily close to the myopically-optimal one.
\end{remark}

\section{Proofs of Performance Bounds}
\label{proofs}

\begin{proof}[Proof of Lemma \ref{lem1}]\
	\begin{enumerate}[(i)]
		\item{We note that, for any control law $u$,
			\begin{equation*}
			\begin{split}& \left\|\phi_u(t_1,x^0)-\phi_u(t_2,x^0)\right\| =\left\|\int_{t_1}^{t_2} f\big(\phi_u(t,x^0)\big)+\sum_{i=1}^m g_i\big(\phi_u(t,x^0)\big)u_i dt\right\| \\
& \leq\int_{t_1}^{t_2}\left\|f\big(\phi_u(t,x^0)\big)+\sum_{i=1}^m g_i\big(\phi_u(t,x^0)\big)u_i\right\|dt \leq\int_{t_1}^{t_2} M_0(m+1)dt=M_0(m+1)|t_2-t_1|\textrm{,}
			\end{split}
			\end{equation*}
			where the last inequality follows from the triangle inequality and because we took $\cU=[-1,1]^m$.
			
			The claim for $x^j$'s follows immediately from the definition: $x^j=\phi_u(j\varepsilon,x^0)$, where $u$ is the control law from Algorithm \ref{alg}.
		}
		\item{Let $j\in\{0,\ldots,m\}$. Then,
			\begin{equation*}
			\begin{split}
			& \left\|\frac{x^{j+1}-x^j}{\varepsilon}- v_{x^{j+1}}(u^j)\right\|\leq \frac{1}{\varepsilon}\Bigg\|\int_0^{\varepsilon}\Bigg(f\big(\phi_{u^j}(t,x^j)\big) +\sum_{i=1}^m g_i\big(\phi_{u^j}(t,x^j)\big)u^j_i-v_{x^{j+1}}(u^j)\Bigg) dt\Bigg\| \\
            & \leq  \frac{1}{\varepsilon}\int_0^{\varepsilon}\Bigg(\left\|f\big(\phi_{u^j}(t,x^j)\big)-f(x^{j+1})\right\| +\sum_{i=1}^m \left\|g_i\big(\phi_{u^j}(t,x^j)\big)u^j_i-\sum_{i=1}^m g_i(x^{j+1})u^j_i\right\|\Bigg)dt \\
& \leq \frac{1}{\varepsilon}\int_0^{\varepsilon}(m+1)M_1\|\phi_{u^j}(t,x^j)-x^{j+1}\|dt \leq\frac{1}{\varepsilon}\int_0^{\varepsilon}M_0M_1(m+1)^2t dt= \frac{M_0M_1(m+1)^2\varepsilon}{2}\textrm{,}
			\end{split}
			\end{equation*}
			where we use the definition of $v_{x^{j+1}}(u^j)$ and the fact that Lipschitz constants of $f,g_i$ are no larger than $M_1$. The last inequality follows from (i).}
		\item{Let $j\in\{0,\ldots,m\}$. We note that 
			$\|v_{x^{j+1}}(u^j)-v_x(u^j)\| \leq \|f(x^{j+1})-f(x)\|+\sum_{i=1}^m\|(g_i(x^{j+1})-g_i(x))u^j_i\| \leq (m+1)M_1\|x^{j+1}-x\|\leq M_0M_1(m+1)^3\varepsilon\textrm{,}$ where the last two inequalities follow from the bounds on Lipschitz constants of $f,g_i$, and from (i), respectively.}
	\end{enumerate}
\end{proof}

\begin{proof}[Proof of Theorem \ref{learapr}]    
From $u=\sum\lambda_j u^j=\sum\lambda_j(u^*+\Delta u^j)$ and $\sum\lambda_j=1$ we obtain $u=u^*+\lambda_1\Delta u^1+\ldots+\lambda_m\Delta u^m$. By definition of $\Delta u^j$ from line 5 of Algorithm \ref{alg}, we have $(u-u^*)_j=\pm\lambda_j\delta$ for all $j\geq 1$.
Hence, $|\pm\lambda_j\delta|<\|u-u^*\|$ for all $j\geq 1$. Since $u,u^*\in\cU=[-1,1]^m$, we obtain $|\lambda_j|\leq 2\sqrt{m}/\delta$ for all $j\geq 1$. By using $\sum\lambda_j=1$ and triangle inequality, we get $|\lambda_0|\leq 1+2m\sqrt{m}/\delta$.

Now, we noted before that $v_x(u)=\sum \lambda_j v_x(u^j)$. Hence, we have $\|v_x(u)-\sum \lambda_j (x^{j+1}-x^j)/\varepsilon\|\leq\sum |\lambda_j|\|v_x(u^j)-(x^{j+1}-x^j)/\varepsilon\|$. Combining parts (ii) and (iii) of Lemma \ref{lem1}, the right hand side of the previous inequality can be bounded by $\sum|\lambda_j| M_0M_1((m+1)^2/2+(m+1)^3)\varepsilon\leq\sum|\lambda_j|2M_0M_1(m+1)^3\varepsilon\leq (1+4m\sqrt{m}/\delta)2M_0M_1(m+1)^3\varepsilon$, where the last inequality was obtained by using bounds on $\lambda_j$ from the previous paragraph.
\end{proof}

\begin{proof}[Proof of Theorem \ref{maith}]
	Let $\overline{u}$ be the optimal control under dynamics \eqref{thesys} at $y$, i.e., $$G(\phi_2|_{[0,T_2]},v_y(\overline{u}))=\max_{u\in\cU} G(\phi_2|_{[0,T_2]},v_y(u))\textrm{.}$$ Then,
	\begin{equation}
	\label{eqthma}
	\begin{split}
	&\left|G\left(\phi_2|_{[0,T_2]},v_y(\overline{u})\right)-G\left(\phi_2|_{[0,T_2]},v_y(u^*)\right)\right| \\
	& \leq\left|G\left(\phi_2|_{[0,T_2]},v_y(\overline{u})\right)- G\left(\phi_1|_{[0,T_1]},\tilde{f}+\sum_{i=1}^m \tilde{g}_iu^*_i\right)\right|
+\Bigg|G\left(\phi_1|_{[0,T_1]},\tilde{f}+\sum_{i=1}^m \tilde{g}_iu^*_i\right)-G\left(\phi_2|_{[0,T_2]},v_y(u^*)\right)\Bigg|\textrm{.}
	\end{split}
	\end{equation}
	We first claim that
	\begin{equation}
	\label{eqthm1}
    \Bigg|G\left(\phi_2|_{[0,T_2]},v_y(\overline{u})\right) -G\left(\phi_1|_{[0,T_1]},\tilde{f}+\sum_{i=1}^m \tilde{g}_iu^*_i\right)\Bigg| \leq Ld\left(\phi_1|_{[0,T_1]},\phi_2|_{[0,T_2]}\right)+L(m+1)M_1\|x-y\|+L\nu\textrm{.}
	\end{equation}
	We note that $\|v_y(\overline{u})-\tilde{f}-\sum_{i=1}^m \tilde{g}_i\overline{u}_i\|\leq\|v_y(\overline{u})-v_x(\overline{u})\|+\|v_x(\overline{u})-\tilde{f}-\sum_{i=1}^m \tilde{g}_i\overline{u}_i\|\leq (m+1)M_1\|y-x\|+\nu$. Hence, 
	\begin{equation}
	\label{eqthm2}
	\begin{split}
	& \left|G\left(\phi_2|_{[0,T_2]},v_y(\overline{u})\right) -G\left(\phi_1|_{[0,T_1]},\tilde{f}+\sum_{i=1}^m \tilde{g}_i\overline{u}_i\right)\right| \\
    & \leq L\Bigg(d\left(\phi_1|_{[0,T_1]},\phi_2|_{[0,T_2]}\right) +\left\|v_y(\overline{u})-\tilde{f}-\sum_{i=1}^m\tilde{g}_i\overline{u}_i\right\|\Bigg) \\
    & \leq Ld\left(\phi_1|_{[0,T_1]},\phi_2|_{[0,T_2]}\right) +LM_1(m+1)\|x-y\|+L\nu\textrm{.}
	\end{split}
	\end{equation}
Analogously to \eqref{eqthm2},
	\begin{equation}
	\label{eqthm3}
\left|G\left(\phi_1|_{[0,T_1]},\tilde{f}+\sum_{i=1}^m \tilde{g}_iu^*_i\right) -G\left(\phi_2|_{[0,T_2]},v_y(u^*)\right)\right|
\leq Ld\left(\phi_1|_{[0,T_1]},\phi_2|_{[0,T_2]}\right) +L(m+1)M_1\|x-y\|+L\nu\textrm{.}
	\end{equation}
		Now, if \eqref{eqthm1} was incorrect, then, by combining its negation with either \eqref{eqthm2} or \eqref{eqthm3}, we obtain that $G(\phi_1|_{[0,T_1]},\tilde{f}+\sum_{i=1}^m \tilde{g}_i\overline{u}_i)>G(\phi_1|_{[0,T_1]},\tilde{f}+\sum_{i=1}^m \tilde{g}_iu^*_i)$ or $G(\phi_2|_{[0,T_2]},v_y(u^*))>G(\phi_2|_{[0,T_2]},v_y(\overline{u}))$, which are both contradictions with the definitions of $u^*$ and $\overline{u}$, respectively. Thus, \eqref{eqthm1} holds, and from \eqref{eqthma}, \eqref{eqthm1}, and \eqref{eqthm3}, we obtain the theorem claim.
\end{proof}
	
\begin{proof}[Proof of Corollary \ref{cor2}]
	Using the same notation as in the proof of Theorem \ref{maith}, we obtain $|G(\phi_2|_{[0,T_2]},v_y(\overline{u})) -G(\phi_2|_{[0,T_2]},v_y(u^*+\tilde{u}))|\leq |G(\phi_2|_{[0,T_2]},v_y(\overline{u})) -G(\phi_2|_{[0,T_2]},v_y(u^*))| +|G(\phi_2|_{[0,T_2]},v_y(u^*)) -G(\phi_2|_{[0,T_2]},v_y(u^*+\tilde{u}))|$.	From Theorem \ref{maith}, we know that the first summand on the right hand side of the above inequality is bounded by $2Ld\left(\phi_1|_{[0,T_1]},\phi_2|_{[0,T_2]}\right)+2LM_1(m+1)\|x-y\|+2L\nu\textrm{.}$ By the definition of $v_y(\cdot)$ and the Lipschitz constant on $G$, the second summand is bounded by $L\|\sum_{i=1}^m g_i(y)\tilde{u}_i\|\leq LM_0m\delta<LM_0(m+1)\delta$. We thus obtain the desired bound.
\end{proof}

\begin{proof}[Proof of Theorem \ref{bigthm}]
	Let us take any $t\geq (m+1)\varepsilon$. Time $t$ is certainly contained in one iteration of the \textbf{repeat} loop in lines 3--14 of Algorithm \ref{alg}, and since $t\geq (m+1)\varepsilon$, this \textbf{repeat} loop is not the first. Let $x^0,\ldots,x^{m+1}$ be the $x^j$'s used in that iteration of the \textbf{repeat} loop. From Theorem \ref{learapr}, we note that the previous iteration resulted in a control input $u^*\in\cU$ which is near-optimal at $x^0$, that is, $u^*$ is optimal for dynamics $u\to\tilde{f}+\sum\tilde{g}_iu_i$, with $\|v_{x^0}(u)-(\tilde{f}+\sum\tilde{g}_iu_i)\|\leq 2M_0M_1(m+1)^3(1+4m^{3/2}/\delta)\varepsilon$ for all $u\in\cU$.
	
	Now, we note that, at time $t$, $u^+(t)=u^*+\Delta u^j$, for some $j\in\{0,\ldots,m\}$ and $\|\Delta u^j\|\leq\delta$. Thus, we can apply Corollary \ref{cor2}, with $\nu=2M_0M_1(m+1)^3(1+4m^{3/2}/\delta)\varepsilon$. We hence obtain 
	\begin{equation}
	\label{eqbig1}
	\begin{split}
	& \Big|\max_u G\left(\phi_{u^+}(\cdot,x_0)|_{[0,t]},v_x(u)\right) - G\left(\phi_{u^+}(\cdot,x_0)|_{[0,t]},v_x(u^+)\right)\Big| \\
	& \leq 2Ld\left(\phi_{u^+}(\cdot,x_0)|_{[0,t]},\phi_{u^+}(\cdot,x_0)|_{[0,t_0]}\right)+2LM_1(m+1)\|x-x^0\| \\ 
    & \textrm{\hskip 10pt} +4LM_0M_1(m+1)^3(1+4m^{3/2}/\delta)\varepsilon + LM_0(m+1)\delta\textrm{,}
	\end{split}
	\end{equation}
	where $t_0$ is the time of the beginning of the current iteration of the repeat loop, i.e., the time when the system was at $x^0$.
	
	We note that by \eqref{defd}, $d\left(\phi_{u^+}(\cdot,x_0)|_{[0,t]},\phi_{u^+}(\cdot,x_0)|_{[0,t_0]}\right)=t-t_0\in[0,(m+1)\varepsilon)\textrm{.}$ Additionally, we note that $\|x-x^0\|=\|\phi_{u^+}(t',x^0)-\phi_{u^+}(0,x^0)\|$ for  $t'=t-t_0\in[0,(m+1)\varepsilon)$. Thus, by Part (i) of Lemma \ref{lem1}, we obtain $\|x-x^0\|\leq M_0(m+1)^2\varepsilon$. Hence, from \eqref{eqbig1}, and taking some very liberal bounds, it now follows that $|\max_u G(x,v_x(u)) -G(x,v_x(u^+))| \leq 6L(M_0+1)(M_1+1)(m+1)^3(1+4m^{3/2}/\delta)\varepsilon+LM_0(m+1)\delta$. The claim of the theorem follows by noting that $d\phi_u(0+,x)/dt=f(x)+\sum g_i(x)u_i=v_x(u)$.
\end{proof}

\section*{Acknowledgements}
This research was partly funded by grants FA8650-15-C-2546 and FA8650-16-C-2610 from the Air Force Research Laboratory (AFRL), grant W911NF-16-1-0001 from the Defense Advanced Research Projects Agency (DARPA), and grant W911NF-15-1-0592 from the Army Research Office (ARO). The authors thank Mohamadreza Ahmadi for his valuable comments in improving this manuscript, as well as for his suggestion of a Van der Pol oscillator as an example in Section~\ref{simul}, and Srilakshmi Pattabiraman for noticing a slight mistake in the previous version of this manuscript.

\bibliographystyle{IEEEtran}
\bibliography{refs}

\begin{thebibliography}{10}
\providecommand{\url}[1]{#1}
\csname url@samestyle\endcsname
\providecommand{\newblock}{\relax}
\providecommand{\bibinfo}[2]{#2}
\providecommand{\BIBentrySTDinterwordspacing}{\spaceskip=0pt\relax}
\providecommand{\BIBentryALTinterwordstretchfactor}{4}
\providecommand{\BIBentryALTinterwordspacing}{\spaceskip=\fontdimen2\font plus
\BIBentryALTinterwordstretchfactor\fontdimen3\font minus
  \fontdimen4\font\relax}
\providecommand{\BIBforeignlanguage}[2]{{%
\expandafter\ifx\csname l@#1\endcsname\relax
\typeout{** WARNING: IEEEtran.bst: No hyphenation pattern has been}%
\typeout{** loaded for the language `#1'. Using the pattern for}%
\typeout{** the default language instead.}%
\else
\language=\csname l@#1\endcsname
\fi
#2}}
\providecommand{\BIBdecl}{\relax}
\BIBdecl

\bibitem{Ornetal17C}
M.~Ornik, A.~Israel, and U.~Topcu, ``Myopic control of systems with unknown
  dynamics,'' 2017, submitted to the 2018 American Control Conference.

\bibitem{Alo06}
S.~Aloni, \emph{Israeli F-15 Eagle Units in Combat}.\hskip 1em plus 0.5em minus
  0.4em\relax Osprey Publishing, 2006.

\bibitem{Bruetal16}
S.~L. Brunton, J.~L. Proctor, and J.~N. Kutz, ``Discovering governing equations
  from data by sparse identification of nonlinear dynamical systems,''
  \emph{Proc. Natl. Acad. Sci. U.S.A.}, vol. 113, no.~15, pp. 3932--3937, 2016.

\bibitem{Cubetal12}
T.~S. Cubitt, J.~Eisert, and M.~M. Wolf, ``Extracting dynamical equations from
  experimental data is {NP} hard,'' \emph{Phys. Rev. Lett.}, vol. 108, 2012.

\bibitem{Hiletal15}
D.~J.~A. Hills, A.~M. Gr\"{u}tter, and J.~J. Hudson, ``An algorithm for
  discovering {Lagrangians} automatically from data,'' \emph{PeerJ Comput.
  Sci.}, vol.~1, 2015.

\bibitem{SchLip09}
M.~Schmidt and H.~Lipson, ``Distilling free-form natural laws from experimental
  data,'' \emph{Science}, vol. 324, no. 5923, pp. 81--85, 2009.

\bibitem{SolOst08}
D.~P. Solomatine and A.~Ostfeld, ``Data-driven modelling: some past experiences
  and new approaches,'' \emph{J. Hydroinf.}, vol.~10, no.~1, pp. 3--22, 2008.

\bibitem{Ahmetal17}
M.~Ahmadi, A.~Israel, and U.~Topcu, ``Safety assessment based on
  physically-viable data-driven models,'' in \emph{56st IEEE Conference on
  Decision and Control}, 2017, accepted.

\bibitem{CopRat94}
R.~P. Copeland and K.~S. Rattan, ``A fuzzy logic supervisor for {PID} control
  of unknown systems,'' in \emph{IEEE International Symposium on Intelligent
  Control}, 1994, pp. 22--26.

\bibitem{Khaetal16}
S.~Khadraoui, H.~N. Nounou, M.~N. Nounou, A.~Datta, and S.~P. Bhattacharyya,
  ``Adaptive controller design for unknown systems using measured data,''
  \emph{Asian J. Control}, vol.~18, no.~4, pp. 1453--1466, 2016.

\bibitem{NaHer14}
J.~Na and G.~Herrmann, ``Online adaptive approximate optimal tracking control
  with simplified dual approximation structure for continuous-time unknown
  nonlinear systems,'' \emph{IEEE/CAA J. Autom. Sin.}, vol.~1, no.~4, pp.
  412--422, 2014.

\bibitem{Yehetal90}
H.-H. Yeh, S.~S. Banda, and P.~J. Lynch, ``Control of unknown systems via
  deconvolution,'' \emph{Dyn. Control}, vol.~13, no.~3, pp. 416--423, 1990.

\bibitem{Zhoetal12}
Z.~Zhou, R.~Takei, H.~Huang, and C.~J. Tomlin, ``A general, open-loop
  formulation for reach-avoid games,'' in \emph{51st IEEE Conference on
  Decision and Control}, 2012, pp. 6501--6506.

\bibitem{Joy12}
D.~Joyce, ``On manifolds with corners,'' in \emph{Advances in Geometric
  Analysis}, S.~Janeczko, J.~Li, and D.~H. Phong, Eds.\hskip 1em plus 0.5em
  minus 0.4em\relax International Press, 2012, pp. 225--258.

\bibitem{Bos92}
J.~T. Bosworth, ``Linearized aerodynamic and control law models of the {X-29A}
  airplane and comparison with flight data,'' National Aeronautics and Space
  Administration, Tech. Rep. NASA TM-4356, 1992.

\bibitem{Duketal88}
E.~L. Duke, R.~F. Antoniewicz, and K.~D. Krambeer, ``Derivation and definition
  of a linear aircraft model,'' National Aeronautics and Space Administration,
  Tech. Rep. NASA RP-1207, 1988.

\bibitem{Steetal16}
B.~L. Stevens, F.~L. Lewis, and E.~N. Johnson, \emph{Aircraft Control and
  Simulation: Dynamics, Controls Design, and Autonomous Systems}.\hskip 1em
  plus 0.5em minus 0.4em\relax Wiley, 2016.

\bibitem{LaV06}
S.~M. LaValle, \emph{Planning Algorithms}.\hskip 1em plus 0.5em minus
  0.4em\relax Cambridge University Press, 2006.

\bibitem{Navetal99}
Y.~Naveh, P.~Z. Bar-Yoseph, and Y.~Halevi, ``Nonlinear modeling and control of
  a unicycle,'' \emph{Dyn. Control}, vol.~9, no.~4, pp. 279--296, 1999.

\bibitem{Moketal06}
A.~Mokhtari, A.~Benallegue, and Y.~Orlov, ``Exact linearization and sliding
  mode observer for a quadrotor unmanned aerial vehicle,'' \emph{Int. J. Rob.
  Autom.}, vol.~21, no.~1, pp. 39--49, 2006.

\bibitem{GarJor77}
W.~L. Garrard and J.~M. Jordan, ``Design of nonlinear automatic flight control
  systems,'' \emph{Automatica}, vol.~13, no.~5, pp. 497--505, 1977.

\bibitem{BulLew04}
F.~Bullo and A.~D. Lewis, \emph{Geometric Control of Mechanical Systems}.\hskip
  1em plus 0.5em minus 0.4em\relax Springer, 2004.

\bibitem{Isi95}
A.~Isidori, \emph{Nonlinear Control Systems}.\hskip 1em plus 0.5em minus
  0.4em\relax Springer, 1995.

\bibitem{Son98}
E.~D. Sontag, \emph{Mathematical Control Theory: Deterministic Finite
  Dimensional Systems}.\hskip 1em plus 0.5em minus 0.4em\relax Springer, 1998.

\bibitem{FadBro06}
F.~Fadaie and M.~E. Broucke, ``A viability problem for control affine systems
  with application to collision avoidance,'' in \emph{45st IEEE Conference on
  Decision and Control}, 2006, pp. 5998--6003.

\bibitem{DorBis11}
R.~C. Dorf and R.~H. Bishop, \emph{Modern Control Systems}.\hskip 1em plus
  0.5em minus 0.4em\relax Prentice Hall, 2011.

\bibitem{KwaSiv72}
H.~Kwakernaak and R.~Sivan, \emph{Linear Optimal Control Systems}.\hskip 1em
  plus 0.5em minus 0.4em\relax Wiley, 1972.

\bibitem{Esfetal16}
P.~M. Esfahani, D.~Chatterjee, and J.~Lygeros, ``The stochastic reach-avoid
  problem and set characterization for diffusions,'' \emph{Automatica},
  vol.~70, pp. 43--56, 2016.

\bibitem{Lev11}
W.~S. Levine, \emph{The Control Systems Handbook, Second Edition: Control
  System Advanced Methods}.\hskip 1em plus 0.5em minus 0.4em\relax CRC Press,
  2011.

\bibitem{Pow11}
W.~B. Powell, \emph{Approximate Dynamic Programming: Solving the Curses of
  Dimensionality}.\hskip 1em plus 0.5em minus 0.4em\relax Wiley, 2011.

\bibitem{RimBro94}
R.~D. Rimey and C.~M. Brown, ``Control of selective perception using {Bayes}
  nets and decision theory,'' \emph{Int. J. Comput. Vision}, vol.~12, no.~2,
  pp. 173--207, 1994.

\bibitem{ZhaYu13}
S.~Zhang and A.~J. Yu, ``Forgetful {Bayes} and myopic planning: Human learning
  and decision-making in a bandit setting,'' in \emph{Neural Information
  Processing Systems Conference}, 2013, pp. 2607--2615.

\bibitem{AlvMar11}
I.~Alvarez and S.~Martin, ``Geometric robustness of viability kernels and
  resilience basins,'' in \emph{Viability and Resilience of Complex Systems},
  G.~Deffuant and N.~Gilbert, Eds.\hskip 1em plus 0.5em minus 0.4em\relax
  Springer, 2011, pp. 193--218.

\bibitem{ManPnu92}
Z.~Manna and A.~Pnueli, \emph{The Temporal Logic of Reactive and Concurrent
  Systems}.\hskip 1em plus 0.5em minus 0.4em\relax Springer, 1992.

\bibitem{Papetal16}
I.~Papusha, J.~Fu, U.~Topcu, and R.~M. Murray, ``Automata theory meets
  approximate dynamic programming: Optimal control with temporal logic
  constraints,'' in \emph{55th IEEE Conference on Decision and Control}, 2016,
  pp. 434--440.

\bibitem{AnsMur16}
A.~R. Ansari and T.~D. Murphey, ``Sequential action control: Closed-form
  optimal control for nonlinear and nonsmooth systems,'' \emph{IEEE Trans.
  Rob.}, vol.~32, no.~5, pp. 1196--1214, 2016.

\bibitem{Ecketal14}
H.-A. Eckel, S.~Scharring, S.~Karg, C.~Illg, and J.~Peter, ``Overview of laser
  ablation micropropulsion research activities at {DLR Stuttgart},'' in
  \emph{International Symposium on High Power Laser Ablation and Beamed Energy
  Propulsion}, 2014.

\bibitem{Krejetal17}
D.~Krejci, F.~Mier-Hicks, R.~Thomas, T.~Haag, and P.~Lozano, ``Emission
  characteristics of passively fed electrospray microthrusters with propellant
  reservoirs,'' \emph{J. Spacecr. Rockets}, vol.~54, no.~2, pp. 447--458, 2017.

\bibitem{Bry94}
A.~E. Bryson, \emph{Control of Spacecraft and Aircraft}.\hskip 1em plus 0.5em
  minus 0.4em\relax Princeton University Press, 1994.

\bibitem{CheChe01}
J.~Che and D.~Chen, ``Automatic landing control using {H}$_\infty$ control and
  stable inversion,'' in \emph{40st IEEE Conference on Decision and Control},
  2001, pp. 241--246.

\bibitem{Ganetal02}
S.~Ganguli, A.~Marcos, and G.~Balas, ``Reconfigurable {LPV} control design for
  {Boeing 747-100/200} longitudinal axis,'' in \emph{American Control
  Conference}, 2002, pp. 3612--3617.

\bibitem{Smi75}
H.~J. Smith, ``A flight test investigation of the rolling moments induced on a
  {T-37B} airplane in the wake of a {B-747} airplane,'' National Aeronautics
  and Space Administration, Tech. Rep. TM X-56031, 1975.

\bibitem{MulSte93}
S.~S. Mulgund and R.~F. Stengel, ``Target pitch angle for the microburst escape
  maneuver,'' \emph{J. Aircr.}, vol.~30, no.~6, pp. 826--832, 1993.

\bibitem{Vid02}
M.~Vidyasagar, \emph{Nonlinear Systems Analysis}.\hskip 1em plus 0.5em minus
  0.4em\relax SIAM, 2002.

\bibitem{Dutetal03}
M.~S. Dutra, A.~C. {de Pina Filho}, and V.~F. Romano, ``Modeling of a bipedal
  locomotor using coupled nonlinear oscillators of {Van der Pol},'' \emph{Biol.
  Cybern.}, vol.~88, no.~4, pp. 286--292, 2003.

\bibitem{HelEsp08}
R.~H\'{e}liot and B.~Espiau, ``Online generation of cyclic leg trajectories
  synchronized with sensor measurement,'' \emph{Rob. Auton. Syst.}, vol.~56,
  no.~5, pp. 410--421, 2008.

\bibitem{VesDem05}
P.~Veskos and Y.~Demiris, ``Developmental acquisition of entrainment skills in
  robot swinging using van der {Pol} oscillators,'' in \emph{Fifth
  International Workshop on Epigenetic Robotics: Modeling Cognitive Development
  in Robotic Systems}, 2005, pp. 87--93.

\bibitem{WirRan02}
S.~Wirkus and R.~Rand, ``The dynamics of two coupled {van der Pol} oscillators
  with delay coupling,'' \emph{Nonlinear Dyn.}, vol.~30, no.~3, pp. 205--221,
  2002.

\bibitem{MouPer05}
E.~Moulay and W.~Perruquetti, ``Stabilization of nonaffine systems: A
  constructive method for polynomial systems,'' \emph{IEEE Trans. Autom.
  Control}, vol.~50, no.~4, pp. 520--526, 2005.

\bibitem{FleRis75}
W.~H. Fleming and R.~W. Rishel, \emph{Deterministic and Stochastic Optimal
  Control}.\hskip 1em plus 0.5em minus 0.4em\relax Springer, 1975.

\bibitem{Mar90}
M.~Mariton, \emph{Jump Linear Systems in Automatic Control}.\hskip 1em plus
  0.5em minus 0.4em\relax CRC Press, 1990.

\bibitem{Bou10}
L.~Boudjenah, ``Existence of solutions to differential inclusions with delayed
  arguments,'' \emph{Electr. J. Differ. Equ.}, vol. 2010, no. 175, pp. 1--8,
  2010.

\bibitem{Car97}
D.~A. Carlson, ``Asymptotic stability for a class of delay differential
  inclusions,'' \emph{J. Differ. Equ.}, vol. 139, no.~2, pp. 219--236, 1997.

\bibitem{HenOua09}
J.~Henderson and A.~Ouahab, ``Fractional functional differential inclusions
  with finite delay,'' \emph{Nonlinear Anal. Theory Methods Appl.}, vol.~70,
  no.~5, pp. 2091--2105, 2009.

\bibitem{Liuetal16}
K.-Z. Liu, X.-M. Sun, J.~Liu, and A.~R. Teel, ``Stability theorems for delay
  differential inclusions,'' \emph{IEEE Trans. Autom. Control}, vol.~61,
  no.~10, pp. 3215--3220, 2016.

\bibitem{Ana75}
B.~I. Anan'ev, ``Teorema sushchestovovaniya dlya differentsial'nogo
  vklyucheniya s peremenn'im zapazd'ivaniem {[An} existence theorem for a
  differential inclusion with variable lag],'' \emph{Differentsial'nye
  uravneniya}, vol.~11, no.~7, pp. 1155--1158, 1975.

\end{thebibliography}

\end{document}